\documentclass[reqno, 11pt]{amsart}
\usepackage{amsmath}
\usepackage{amsfonts}
\usepackage{amssymb}
\usepackage{hyperref}
\usepackage{mathrsfs}
\usepackage{amsthm}
\newtheorem{theorem}{Theorem}[section]
\newtheorem{lemma}[theorem]{Lemma}
\usepackage{tikz}
\usetikzlibrary{matrix,calc}
\usepackage{tikz-cd}
\usepackage{epsfig}  		% For postscript
\usepackage{epic,eepic}
\usepackage{stmaryrd}  
\usepackage{amscd}
\usepackage{amsbsy}
\usepackage{commath, longtable}
\usepackage{comment, enumerate}
\usepackage{geometry}
\usepackage[matrix,arrow]{xy}
\usepackage{color}
\usepackage{mathtools,caption}
\usepackage{tikz-cd}
\usepackage{longtable}
\usepackage[utf8]{inputenc}
\usepackage[OT2,T1]{fontenc}
\usepackage[normalem]{ulem}
\usepackage[customcolors]{hf-tikz}
\usepackage{booktabs}
\hypersetup{
	colorlinks=true,
	linkcolor=red,
	filecolor=blue,
	citecolor=blue,
	urlcolor=green,
	pdftitle={2 Selmer Companion Modular forms},
}
\newtheorem{prop}[theorem]{Proposition}

\newtheorem{definition}[theorem]{Definition}

\newtheorem{remark}[theorem]{Remark}
\newtheorem{conjecture}{Conjecture}

\newtheorem{innercustomhypo}{Hypothesis}
\newenvironment{customhypo}[1]
{\renewcommand\theinnercustomhypo{#1}\innercustomhypo}
{\endinnercustomhypo}

\DeclareMathOperator{\Hom}{Hom}

\DeclareMathOperator{\Gal}{Gal}

\DeclareMathOperator{\GL}{GL}
\DeclareMathOperator{\SL}{SL}

\DeclareMathOperator{\Sel}{Sel}

\DeclareMathOperator{\Aut}{Aut}

\DeclareMathOperator{\Frob}{Frob}

\newcommand{\ZZ}{\mathbb Z}
\newcommand{\CC}{\mathbb C}
\newcommand{\Zp}{\mathbb{Z}_p}
\newcommand{\Qp}{\mathbb{Q}_p}

\newcommand{\NN}{\mathbb N}

\newcommand{\EC}{\mathsf E}
\newcommand{\cK}{\mathcal K}

\newcommand{\cO}{\mathcal O}

\newcommand{\p}{\mathfrak p}

\newcommand{\Q}{{\mathbb{Q}}}

\newcommand{\F}{{\mathbb{F}}}

\newcommand{\Afc}{A_{f\chi}}
\newcommand{\tx}[1]{\text{#1}}

\title{$2$-Selmer Companion modular forms}
\author[Abhishek]{Abhishek }
\author[S.~Jha]{Somnath Jha }
\author[S.~Shekhar]{Sudhanshu Shekhar }

\address[]{Department of Mathematics and Statistics\\ IIT Kanpur\\ India, 208016}

\email{abhi.math04@gmail.com, jhasom@iitk.ac.in, sudhansh@iitk.ac.in, }

\begin{document}

	\begin{abstract}
    Let $N$ be a positive integer and $K$ be a number field. 
	Suppose that $f_1, f_2\in S_k(\Gamma_0(N))$ are two newforms such that their residual {Galois representations} at $2$ are isomorphic. Let $\omega_2:G_\Q\rightarrow \ZZ_2^\ast$ be the $2$-adic cyclotomic character. Then, under  suitable hypotheses, we have shown that for every quadratic character $\chi$ of $K$ and each critical twist $j$,  the residual  Greenberg $2$-Selmer groups of $f_1\chi\omega_2^{-j}$ and $f_2\chi\omega_2^{-j}$ over $K$ are isomorphic. This generalizes the corresponding result of Mazur-Rubin on $2$-Selmer companion elliptic curves. Conversely,  
     if the difference of the residual  Greenberg (respectively Bloch-Kato) $2$-Selmer ranks of $f_1\chi$ and $f_2\chi$ is bounded independent of every quadratic character $\chi$ of $K$, then under suitable hypotheses we have shown that the residual Galois representations at $2$ of $f_1$ and $f_2$ are isomorphic as $G_K$-modules. The corresponding result for elliptic curves was a conjecture of Mazur-Rubin, which was proved by M. Yu.

   		    \end{abstract}
	\maketitle
	\section*{Introduction}\label{Intro}
	The $n$-Selmer group of an elliptic curve captures important arithmetic properties of the elliptic curve and there has been a considerable interest in understanding the $n$-Selmer groups of elliptic curves over the years. In particular, Mazur-Rubin have studied the rank distribution of $n$-Selmer groups of an elliptic curve in a family of quadratic twists of the given elliptic curve. Let $p$ be a  prime and $\EC_1, \EC_2$ be two elliptic curves defined over a number field $K$.    In \cite{mr}, Mazur-Rubin  have the interesting formulation that the congruence between two elliptic curves $\EC_1, \EC_2$ at  $p$ over $K$ can be understood by comparing the $p$-Selmer groups of the quadratic twists $\EC_1^\chi$ and $\EC_2^\chi$,   for all possible quadratic characters $\chi$ of $K$.

	 Let $\Sel_p(\EC/K)$ denote the classical  $p$-Selmer group of an elliptic curve $\EC$ over $K$. We recall the notion of $p$-Selmer companion curves following \cite{mr}. 
	\begin{definition}\cite[Definition 1.2]{mr}
		Two elliptic curves $\EC_1$ and $\EC_2$ are said to be $p$-Selmer companion curves over a number field $K$, if for every quadratic character $\chi$ of $K$, there is a group isomorphism between the  $p$-Selmer groups of $\EC_1^\chi$ and $\EC_2^\chi$, i.e. 
		$\Sel_p(\EC_1^\chi/K)\cong\Sel_p(\EC^\chi_2/K).$    
	\end{definition}
	
	In \cite[Theorem 3.1]{mr}, authors provide a set of sufficient conditions for $\EC_1$ and $\EC_2$ to be $p$-Selmer companion curves for every prime $p$. The case $p=2$ is of particular interest to us and can be stated as follows:
	\begin{theorem}\cite{mr}
		Let $\EC_1$ and $\EC_2$ be two elliptic curves defined over $K$ and suppose that $\EC_1, \EC_2$ have potentially multiplicative reduction at every prime $\p$ above $2$. For $i=1,2$, let $S_i$ be the set of primes of $K$ where $\EC_i$ has potentially multiplicative reduction. Assume that there is a $G_K$-isomorphism  $\phi:\EC_1[4]\rightarrow\EC_2[4]$ and $S_1=S_2$. Moreover, suppose that for every $\ell\in S_1=S_2$, the isomorphism $\phi$ takes $\mathcal{C}_{\EC_1/K_\ell}[4]$ to $\mathcal{C}_{\EC_2/K_\ell}[4]$, where $\mathcal{C}_{\EC_i/K_\ell}[4]$ is as defined in \cite[Definition 2.3]{mr}.  
		Then $\EC_1$ and $\EC_2$ are $2$-Selmer companions over every number field containing $K$. 
	\end{theorem}
	
	In \cite{JSM19}, \cite[Theorem 3.1]{mr} was extended to modular forms for an odd prime. If a modular form is $p$-ordinary, then they considered  $p^r$-Greenberg Selmer group (see \cite[Theorem 3.1]{JSM19}). In the general case (including non-ordinary at $p$), they considered $p^r$-Bloch-Kato Selmer groups (see \cite[Theorem 4.10]{JSM19}). However, the case $p=2$ was left out in \cite{JSM19}. In the first part of this article for $p=2$, we extended \cite[Theorem 3.1]{mr} for $p$-ordinary modular forms.  
	
	Let $p=2$ and $f_1, f_2\in S_{k}(\Gamma_1(Np^t))$ be two weight $k$, normalized cuspidal Hecke eigenforms which are  newforms of level $Np^t$. Let $\cK$ be the number field generated by the Fourier coefficients of $f_1$ and $f_2$. {Let $\pi\mid p$ be a prime of $\cK$ induced by the embedding $i_p:\bar{\Q}\hookrightarrow \bar{\Q}_p$ and $\cK_\pi$ denote the completion of $\cK$ at $\pi$. We fix an uniformizer, again denoted by $\pi$,  of the ring of integers $\cO=\cO_{\cK_\pi}$ of the completion of $\cK_\pi$. (By a slight  abuse of notation,  $\pi$ is used here for both the prime and the uniformizer, with the distinction being clear from the context.)}  We assume that $f_1, f_2$ are $p$-ordinary i.e $i_p(a_p(f_1)) \text{ and } i_p(a_p(f_2))$ are $p$-adic units. For an $\cO$-module $M$ and $a\in \cO$, set $M[a]:=\{m\in M:a m=0\}$. For $f\in\{f_1, f_2\}$, let $\rho_f:G_\Q\rightarrow \Aut(V_f)$ be the Galois representation associated to $f$, where $V_f$ is a $2$-dimensional $\cK_\pi$-vector space. Fix a $G_\Q$-stable lattice $T_f$ of $V_f$ and set $A_f:=\frac{V_f}{T_f}$. We denote by $\bar{\rho}_f$, the associated residual Galois representation at $p$. Let $K$ be a number field and $\p\mid p$ be a prime of $K$. Let $S_{Gr}(A_{f}[\pi]/K)$ denote the Greenberg $\pi$-Selmer group of $f$ over $K$ (definition \ref{pi.sel.gp}). The first main result of the article is theorem \ref{mainthm}. Let $\omega_p:G_\Q\rightarrow \ZZ_p^\ast$ be the $p$-adic cyclotomic character. Some of the notation and hypotheses used in theorem \ref{mainthm} are defined in \S\ref{pf of main thm}. 
	\begin{theorem}\label{mainthm}
		Let $p=2$, $N$ be an odd  square-free positive integer, $k$ be an even positive integer and $K$ be a number field. Let $f_1,\, f_2\in S_k(\Gamma_0(Np^t))$ be two $p$-ordinary normalized cuspidal Hecke eigenforms which are  newforms of level $Np^t$ with $t\ge 0$. Assume that hypothesis $\mathrm{\ref{I}}$ holds i.e. $\cK_\pi/\Q_2$ is an unramified extension. Further assume for every prime $v\mid N$,  hypothesis $\mathrm{\ref{C2}}$ holds for both $f_1$ and $f_2$ i.e. $cond_v(\rho_{f_i})=cond_v(\bar{\rho}_{f_i})$, for $i=1, 2$. Also suppose  either ($i$) or $(ii)$ holds:
		\begin{enumerate} 
			\item[$(i)$] Both the conditions, stated below, hold:
            \begin{enumerate}
            \item We have a $G_K$-isomorphism $\phi:A_{f_1}[\pi^2]\rightarrow A_{f_2}[\pi^2]$.
                 \item $\omega_{p|{G_\p}}$ is non-trivial $\mod{\pi^2}$, i.e. $\omega_{p|_{G_\p}}\not\equiv 1\pmod{\pi^2}$, for every prime $\p\mid p$ of $K$.
            \end{enumerate}
            
			\item[$(ii)$] All  three conditions, stated below,  hold: 
			\begin{enumerate}
				\item 
				 We have a $G_K$-isomorphism $\phi:A_{f_1}[\pi]\rightarrow A_{f_2}[\pi]$.

				\item $K$ has no real place. 
				\item $\bar{\rho}_{f_1}$ is ramified at every prime $\p\mid p$ of $K$.  
			\end{enumerate} 
		\end{enumerate}
		Then for every quadratic character $\chi$ of $K$ and for every $j$ with $0\le j\le k-2$, we have an isomorphism of Greenberg $\pi$-Selmer groups 
		$S_{Gr}(A_{f_1\chi(-j)}[\pi]/K)\cong S_{Gr}(A_{f_2\chi(-j)}[\pi]/K).$
	\end{theorem}
\noindent Note that the condition $(i)(b)$ always holds for $K=\Q$. 
    
	\begin{remark}\label{rmk0.4} 
		\textnormal{A crucial assumption in \cite[Theorem 3.1]{mr}
		for $p=2$ case is that $\EC_1, \EC_2$ have potentially multiplicative reduction at every prime $\p\mid 2$. In particular when $p=2$,  good ordinary and (good) supersingular reduction at $p=2$ cases are not covered in \cite{mr}. For an odd prime $p$, \cite[Theorem 3.1]{mr} was generalised in \cite{JSM19} for the Greenberg $p$-Selmer groups of $p$-ordinary modular forms. Further it was extended in \cite{JSM19} to the Bloch-Kato $p$-Selmer groups by comparing the Greenberg and the Bloch-Kato $p$-Selmer groups. An important assumption used there for the comparison was $a_p(f)\ne \pm 1\pmod{p}$. In our case when $p=2$, for $p$-ordinary modular forms, we always land in the case $a_p(f)=1\pmod{p}$. There seems to be a scarcity of relevant literature for studying the residual Bloch-Kato $2$-Selmer groups of modular forms. In particular, for Bloch-Kato $2$-Selmer group of  modular forms, the  local condition $H^1_{BK}(G_\p, A_{f}[\pi])$(see definition \ref{BK.Loc.cond}) with $\p \mid 2$,  seems to be difficult to determine in terms of the residual Galois representation at $2$.   In \cite{JSM19} for an odd prime, in the non-ordinary case i.e. when $p\mid a_p(f)$, the Bloch-Kato $p$-Selmer groups were studied by comparing it with the signed Selmer groups of \cite{llzant}. Since the results in \cite{llzant} are proved for an odd prime, for $p=2$, it can not be applied to non-ordinary  at $p$ modular forms   to study the Bloch-Kato $p$-Selmer groups.}
	\end{remark}
    
\begin{remark}
\textnormal{Note that the $G_K$-isomorphism $\phi:A_{f_1}[\pi]\rightarrow A_{f_2}[\pi]$ and the Selmer groups $ S_{Gr}(A_{f_i}[\pi]/K)$ depend on the choices of Galois invariant lattices $T_{f_i}\subset V_{f_i}$. So naturally,  the definition of Selmer companion modular forms (definition \ref{pi.sel.gp}) and also the result in theorem \ref{mainthm} depend on the choices of Galois invariant lattices. If the residual Galois representation $\bar{\rho}_f$ of a newform $f$ is an absolutely irreducible $G_K$-module, then there is {a}
 unique $G_K$-invariant lattices $T_f$ (upto homothety), see \cite{epw}.}   
\end{remark}

	Given two congruent mod-$p$ elliptic curves or modular forms, the basic idea to compare their $p$-Selmer groups is to express each of the $p$-Selmer groups in terms of the respective residual Galois representation \cite{mr, JSM19}.    We follow a similar broad strategy to prove theorem \ref{mainthm} here. However the situation when $p=2$ is more delicate and we encounter various additional difficulties and the various steps in the proof of theorem \ref{mainthm} have different flavours compared to  \cite{JSM19}. As a consequence, our result for the even prime in this article, is obtained under more restrictive hypotheses.
	
	We make the simplifying  assumption that $\cK_\pi/\Q_2$ is unramified, which enables us to have $A_f[\pi]=A_f[2]$.  This is used throughout. To arrive at theorem $\ref{mainthm}$, we do a prime by prime comparison of the local factors appearing in the Greenberg $p$-Selmer group of $f_1$ and $f_2$ (definition \ref{pi.sel.gp}, \eqref{Gr.Loc.cond}). 
    
    For $\p\mid p$, we are unable to determine the local factor $H^1_{Gr}(G_\p, A_{f\chi(-j)}[\pi])$ purely in terms of the residual Galois representation $A_{f\chi(-j)}[\pi]$ (unlike in \cite{JSM19}). However, we circumvent this by showing that (lemma \ref{Lemma 3}) under the isomorphism $\tilde{\phi}:H^1(G_\p, A_{f_1\chi(-j)}[\pi])\rightarrow H^1(G_\p, A_{f_2\chi(-j)}[\pi])$ induced from the $G_K$-isomorphism $\phi:A_{f_1}[\pi]\rightarrow A_{f_2}[\pi]$, the local condition $H^1_{Gr}(G_\p, A_{f_1\chi(-j)}[\pi])$ maps onto $H^1_{Gr}(G_\p, A_{f_2\chi(-j)}[\pi])$.

	\begin{remark}
    \textnormal{In the proof of \cite{JSM19}, for a prime $\ell\ne p$, a crucial assumption was  the conductor hypothesis for $f$ i.e. $cond_\ell(\rho_f)=cond_\ell(\bar{\rho_f})$ (see hypothesis \ref{C2}). However, for $p=2$ and in the setting of our theorem, the hypothesis \ref{C2} may not hold if the nebentypus of $f$ is non-trivial. Hence we require that $f_1, f_2\in S_k(\Gamma_0(Np^t))$ in theorem \ref{mainthm}. 
    {For example, let $p=2$  and $N$ be an odd square-free integer. Let $f\in S_k(\Gamma_0(Np^t), \epsilon)$ be a newform of weight $k$ and nebentypus $\epsilon$, where $\epsilon$ is a quadratic character of conductor $C\mid N$. Assume that $\cK=\Q(\set{a_n(f)\colon n\in \NN})=\Q$.  Then for a rational prime $\ell\mid\mid C$, we have $\rho_f|_{G_\ell}\sim \begin{pmatrix}
				\epsilon & 0 \\ 
				0 & 1
			\end{pmatrix},$ see \cite[Theorem 3.26(3(a))]{Hida}. Therefore $cond_\ell(\rho_f)=\ell$ and since $\epsilon$ is a quadratic character, $\bar{\rho}_f|_{G_\ell}\sim \begin{pmatrix}
				1 & 0 \\ 
				0 & 1
			\end{pmatrix}$. Therefore $cond_\ell(\bar{\rho}_f)=1.$ Thus hypothesis \ref{C2} fails in this case.}   }
		
		\textnormal{Further, it was shown in \cite{JSM19} that the conductor hypothesis for $f$  implies the conductor hypothesis for $f\otimes\chi$, for every quadratic character $\chi$ of $K$. Although in the our case of even prime, the hypothesis \ref{C2} fails for $f\otimes\chi$ for some quadratic character $\chi$ of $K$, we get around this by using an entirely different (from \cite{JSM19}) argument.}   
	\end{remark}
	We prepare the proof of theorem \ref{mainthm} {with} several lemmas. Let $M$ denote the conductor of $\chi$ and define $\Sigma=\Sigma_\chi:=\{v \text{ prime in }K:v\mid NM2\infty\}$. In lemma \ref{modfi.se.gp}, we give an alternative definition of $S_{Gr}(A_{f\chi(-j)}[\pi]/K)$  by replacing $\frac{H^1(G_v, A_{f\chi(-j)}[\pi])}{H^1_{Gr}(G_v, A_{f\chi(-j)}[\pi])}$ with  $\frac{H^1(I_v, A_{f\chi(-j)}[\pi])}{Im(\kappa_{f\chi(-j), v}^{ur})}$ for $v\nmid p, v\in \Sigma$, where $\kappa_{f\chi(-j), v}^{ur}:A_{f\chi(-j)}^{I_v}/\pi \rightarrow H^1(I_v, A_{f\chi(-j)}[\pi])$ denotes the `Kummer map'  \eqref{Sec 3.1}.  A similar expression is also obtained for $\p\mid 2$ by replacing $A_f$ with $A_f^-$, where $A_f^-$ is the quotient of $A_f$ defined using the $p$-ordinary filtration of $A_f$ \eqref{fil. of Af}. In lemma \ref{Lemma 2}, we show that for $v\mid N$, $Im(\kappa_{f\chi(-j), v}^{ur})$  is determined by $A_{f\chi(-j)}[\pi]$. The corresponding results for $\p\mid 2$ case have been discussed in lemmas \ref{Lemma 4}, \ref{Lemma 3}. 
    
    Note that when $p=2$, in the definition of the Selmer group, we need to consider the contribution of the local Galois cohomology groups at infinite places (unlike in \cite{JSM19}). In lemmas \ref{Lemma 5}, \ref{Lemma 4.6} for $v\mid \infty$, we show that  $Im(\kappa_{f\chi(-j), v})$ only depends on $A_{f\chi(-j)}[\pi]$.  Combining all these results, the proof  of theorem \ref{mainthm} is completed in \S \ref{pf of main thm}. We illustrate theorem \ref{mainthm} through a numerical example in \ref{examp.}.       
	
	\medskip
	
Mazur-Rubin \cite{mr} have also introduced a notion of Selmer near-companion curves as follows:
\begin{definition}\cite[Definition 7.12]{mr}
	Let $p$ be a prime and $K$ be a number field.	Two elliptic curves $\EC_1$ and $\EC_2$ are said to be $p$-Selmer near-companion curves over $K$, if there exist a constant $C=C(\EC_1, \EC_2, K)>0$ such that, for every quadratic character $\chi$ of $K$,{\small \begin{center}
	    $\abs{\dim_{\F_p}\Sel_p(\EC_1^\chi/K)-\dim_{\F_p}\Sel_p(\EC_2^\chi/K)}\le C.$\end{center}}
	\end{definition}

An analogue of  \cite[Theorem 3.1]{mr} for Selmer near-companion curves  for $p=2$ is the following: 
	\begin{theorem}\cite[Theorem 7.13]{mr}\label{0.4}
		Let $\EC_1$ and $\EC_1$ be elliptic curve over $K$. Suppose that there is a $G_K$-isomorphism $\EC_1[4]\cong \EC_2[4]$. Then $\EC_1$ and $\EC_2$ are $2$-Selmer near-companion curves over $K$.   
	\end{theorem} 
    We have a generalization of theorem \ref{0.4} for modular forms in proposition \ref{mainthm gr nr com.}. 
Further, the following conjecture was proposed in \cite{mr} as a converse of theorem \ref{0.4}.  
	\begin{conjecture}\cite[Conjecture 7.15]{mr}\label{conj 1}
		If $n\ge 2$ and $\EC_1, \EC_2$ are $n$-Selmer near-companion curves over $K$, then $\EC_1[n]$ is $G_K$-isomorphic to $\EC_2[n]$. 
	\end{conjecture}
	At present, for $n>2$,  conjecture \ref{conj 1} seems to be out of reach. However for $n=2$, M.Yu \cite[Theorem 2]{Myu} has proved conjecture \ref{conj 1}.  In the second main result of this article, we discuss an analogue of \cite[Theorem 2]{Myu} for certain modular forms of (even) weight $k$ and trivial nebentypus. 
	\begin{theorem}\label{Gr nr com. converse}
			Let $p=2$, $N$  be an odd positive integer, $k \in \NN$ be even and $K$ be a number field. Let $f_1,\, f_2\in S_k(\Gamma_0(Np^t))$ be two $p$-ordinary normalized cuspidal Hecke eigenforms which are  newforms of level $Np^t$ with $t\ge 0$, such that $\cK_\pi=\Q_2$ (i.e. hypothesis \ref{III} holds). Assume that  either one of the  residual Galois representations,  $\bar{\rho}_{f_1}$ or $\bar{\rho}_{f_2}$ is an irreducible $G_K$-module. If $f_1$ and $f_2$ are Greenberg $\pi$-Selmer near-companion forms over $K$, then there exist a $G_K$-isomorphism $A_{f_1}[\pi]\cong A_{f_2}[\pi]$. 
	\end{theorem}

   Note that (as there are no obvious tools to attack conjecture \ref{conj 1} for an odd prime)  there is no discussion related to conjecture \ref{conj 1} in \cite{JSM19} and  there is no analogue of theorem \ref{Gr nr com. converse} in \cite{JSM19}.

 We now discuss the main ideas in the proof of theorem \ref{Gr nr com. converse}. We prove the contrapositive statement i.e. if $A_{f_1}[\pi]\ncong A_{f_2}[\pi]$ as $G_K$-module, then we  show that $f_1$ and $f_2$ are not Greenberg $\pi$-Selmer near-companion over $K$. First note that for $i=1, 2$, the hypothesis \ref{III} on $f_i$ gives us that  $\Gal(K(A_{f_i}[\pi])/K)\hookrightarrow S_3$ and then it follows that $K_1=K(A_{f_1}[\pi])=K(A_{f_2}[\pi])=K_2$ if and only if $A_{f_1}[\pi]\cong A_{f_2}[\pi]$ as $G_K$-module. Let $g_i(X)$ be the cubic irreducible polynomial over $K$ such that $K_i$ is the splitting field of $g_i(X)$, for $i=1,2 $. Consider the elliptic curve $\EC_i$ defined by  $\EC_i:Y^2=g_i(X)$. Then $K_i=K(\EC_i[2])$ and  $A_{f_i}[\pi]\cong \EC_i[2]$ as $G_K$-module. We emphasise that the residual Selmer group of $f_i$ at $\pi$ is a priori  not determined by $A_{f_i}[\pi]\cong \EC_i[2]$. So we need to do more work, in addition to using results from \cite{Myu}, to arrive at theorem \ref{Gr nr com. converse}.

 We construct in proposition \ref{modular, ecc sel com}, certain infinite set $\mathfrak{X}$ (see  \eqref{def X})  of quadratic characters $\chi$ of $K$  such that $\abs{\dim_{\F_2}S_{BK}(A_{f_i\chi}[\pi]/K)-\dim_{\F_2}\Sel_2(\EC_i/K)}$ is bounded independent of $\chi\in\mathfrak{X}$ for $i=1,2$. Here $S_{BK}(A_{f_i\chi}[\pi]/K)$ denotes the Bloch-Kato $\pi$-Selmer group of $f_i\chi$ over $K$ (see definition \ref{pi.sel.gp}). 
 Combining proposition \ref{modular, ecc sel com} with \cite[Theorem 2]{Myu}, we arrive at theorem \ref{BK-near-mainthm}, which is an analogue of theorem \ref{Gr nr com. converse}, obtained by replacing the Greenberg $\pi$-Selmer group with the Bloch-Kato $\pi$-Selmer group. From theorem \ref{BK-near-mainthm}, we go on to deduce theorem \ref{Gr nr com. converse} by comparing the Bloch-Kato  and the Greenberg $\pi$-Selmer groups of $f_i$ for $i=1,2$. Throughout \S\ref{converse}, we use the hypothesis that $\bar{\rho}_{f_2}$(say) is irreducible.
 \begin{remark}
\textnormal{   Recall $p$ is equal to $2$ in theorem \ref{Gr nr com. converse}. As we have considered elliptic curves $\EC_i/K$ such that $\EC_i[2]\cong A_{f_i}[\pi]$ as $G_K$-modules and used the results from \cite{Myu}, we need to go through $S_{BK}(A_{f_i\chi}[\pi]/K)$ in the proof of theorem \ref{Gr nr com. converse}. As mentioned in remark \ref{rmk0.4}, for the Bloch-Kato $\pi$-Selmer group, it seems difficult to determine  $H^1_{BK}(G_\p, A_{f_i\chi}[\pi])$ with $\mathfrak p \mid p$  in terms of $A_{f_i\chi}[\pi]$ and hence our result in theorem \ref{mainthm} is restricted to Greenberg $\pi$-Selmer group.   
   However for establishing the converse (theorem \ref{Gr nr com. converse}), we only need to compare the local conditions for the Greenberg and Bloch-Kato $\pi$-Selmer groups at  certain set of infinitely many primes away from $Np\infty$, which is feasible for us. Thus we have the converse result for both Greenberg (theorem \ref{Gr nr com. converse}) and Bloch-Kato (theorem \ref{BK-near-mainthm}) $\pi$-Selmer groups.}
\end{remark}
	\subsection*{Structure of the article:} The \S \ref{Premi.} is preliminary in nature, where we recall the Galois representations attached to modular forms and the definitions of various Selmer groups. In \S\ref{pf of main thm}, we discuss the proof of our first main result (theorem \ref{mainthm}). We prove our second main result (theorem \ref{Gr nr com. converse})  in \S\ref{converse}.  
\section*{Acknowledgement}	
	We thank the anonymous referee for various helpful comments and suggestions for improvements. Abhishek is supported by institute fellowship  at IIT Kanpur.      
     Somnath Jha acknowledges the support of ANRF grant CRG/2022/005923. 
	
	%%%%%%%%%%%%%%%%%%%%%%%%%%%%%%%%%%%%%%%%%%%%%%%%%%%%%%%%%%%%%%%%%%%%%%%%%%%%%%%%%%%%%%%%%%%%%%%%%%%%%%%%%%%%%%
	
	\section{Preliminaries}\label{Premi.}

	Notation: Let  $p$ be a prime. (The results of this section are valid for any prime $p \geq 2$ and we apply them for $p=2$ in \S\ref{pf of main thm} and \S\ref{converse}). %From now on, we fix the prime $p=2$.  
 Throughout we fix an algebraic closure $\overline{\Q}$ of $\Q$ in $\CC$ and  embeddings $i_\infty:\overline{\Q}\rightarrow\CC$ and  $i_\ell:\overline{\Q}\rightarrow\overline{\Q}_\ell$, for every finite rational prime $\ell$. For a number field $K$, $G_K:=\Gal(\bar{K}/K)$
denotes the absolute Galois group of $K$. For a finite set $\Sigma$ of primes of $K$, $G_{K, \Sigma}$ denotes the Galois group of the maximal algebraic extension of $K$ unramified outside $\Sigma$.  We denote the decomposition subgroup of $G_K$ at a prime $v$ of $K$ by $G_v$ and $I_v$ will denote the inertia subgroup of $G_v$.

	\subsection{Galois representation attached to a modular form}\label{2-adic Galois rep.}
	 Let $N, k \in \NN$ with $k$ even and $(p, N)=1$. Let $f$ = ${\sum}_{n\in \NN}a_n(f)q^n\in S_k(\Gamma_0(Np^t))$ be a normalized cuspidal Hecke eigenform which is a newform of level $Np^t$ with $t \geq 0$. Let $L$ be a number field containing $\cK_f:=\Q(\{a_n(f):n\in \NN\})$. (Note that in theorems \ref{mainthm} and \ref{Gr nr com. converse},  we have $f_1, f_2\in S_k(\Gamma_0(Np^t))$ and we will choose $L=\cK_{f_1, f_2}=:\Q(\{a_n(f_1), a_n(f_2):n\in \NN\})$ for both $f_1$ and $f_2$.)   Let $\pi$ be a prime in $\cO_L$ lying above $p$ induced by the embedding $i_p$ and we also denote by $\pi$, a fixed uniformizer of the ring of integers $\cO_{L_\pi}$ of $L_\pi$. We say $f$ is  $p$-ordinary if $i_p(a_p)$ is a $p$-adic unit. Let $\omega_p:G_\Q\rightarrow \ZZ_p^\ast$ be the $p$-adic cyclotomic character. The following theorem is due to Eichler, Shimura, Deligne, Mazur–Wiles, Wiles and others (see \cite[Theorem 3.26]{Hida}).
	\begin{theorem}\label{gal-rep of f}
		Let $f=\underset{n\in \NN}{\sum}a_n(f)q^n\in S_k(\Gamma_0(Np^t))$ as above.  There exists a Galois representation $\rho_f:G_\Q\rightarrow \GL_2(L_\pi)$,  such that
		\begin{enumerate}
			\item $\rho_f$ is unramified at all prime $\ell\nmid Np$ and $\rho_f$ is odd i.e. $\det(\rho_f(c))=-1$, for each complex conjugation $c$. Moreover, the characteristic polynomial of $\rho_f(\Frob_\ell)$, for $\ell\nmid Np$, is given by
			\[\det(1-\rho_f(\Frob_\ell)T)=1-a_\ell(f)T+\omega_p(\Frob_\ell)^{k-1}.\]
			\item(Deligne, Mazur-Wiles) Suppose that $f$ is ordinary at $p$. Let $\delta_{f}$ be the unramified character with $\delta_{f}(\Frob_p)=\alpha_p(f)$, where $\alpha_p(f)$ is the $p$-adic unit root of $X^2-a_p(f)X+p^{k-1}$. Then, 
	\begin{small}		\begin{align*}
				\rho_f|_{G_p}\sim \begin{pmatrix}
					\delta_{f}^{-1}\omega_p^{k-1} & \ast \\
					0 & \delta_{f}
				\end{pmatrix}.
			\end{align*}\label{p_mid p}
            \end{small}
			\item(Langlands, Carayol) Let $\ell\ne p$ be a rational prime. Suppose that $\ell\mid \mid N$, then
	{\small		\[\rho_f|_{G_\ell}\sim \begin{pmatrix}
				\eta_f\omega_p & \ast \\ 
				0 & \eta_f
			\end{pmatrix}, \text{ where $\eta_f:G_\ell\rightarrow \ZZ_p^\times$  is an unramified character.}\]}
			\label{p_mid N}
		\end{enumerate}
	\end{theorem}
Let $V_f\cong L_\pi\oplus L_\pi$ denote the representation space associated with the Galois representation $\rho_f$. Fix a $G_\Q$-invariant lattice $T_f=\cO_{L_\pi}\oplus \cO_{L_\pi}$ and {let $\bar{\rho}_f:G_\Q\rightarrow \GL_2(\frac{\cO_{L_\pi}}{\pi_{L}})$ denote the residual Galois representation of $\rho_f$ associated to $T_f$.}	
	\subsection{Selmer groups of a modular form}\label{Sec 3.1}
	Let $f\in S_k(\Gamma_0(Np^t))$ be a newform   in the setting of \S\ref{2-adic Galois rep.}.	Let $V_f\cong L_\pi \oplus L_\pi$ denote the representation space associated to the Galois representation $\rho_f$. Fix a $G_K$-invariant, $\cO_{L_\pi}$-lattice $T_f\cong \cO_{L_\pi}\oplus\cO_{L_\pi}$ of $V_f$ and set $A_f:=V_f/T_f$. We have the exact sequence $0\rightarrow T_f\rightarrow V_f\stackrel{P_f}{\rightarrow} A_f\rightarrow0$.
	Let $\chi$ be a quadratic character of $K$ and for $j\in \ZZ$, set $V_{f\chi(-j)}:=V_f\otimes\chi\omega_p^{-j}$ and $T_{f\chi(-j)}:=T_f\otimes\chi\omega_p^{-j}$ and $A_{f\chi(-j)}:=\frac{V_{f\chi(-j)}}{T_{f\chi(-j)}}$. Let $M:=cond(\chi)$ and set
		$\Sigma:=\{v \text{ prime of }K: v\mid NMp\infty\}$. For each $v\in\Sigma$, we choose a subgroup $H^1_{\dagger}(G_v, A_{f\chi(-j)})\subset H^1(G_v, A_{f\chi(-j)})$. Let $i:A_{f\chi(-j)}[\pi]\rightarrow A_{f\chi(-j)}$ be the inclusion map and for any prime $v$ of $K$, let $i_v:H^1(G_v, A_{f\chi(-j)}[\pi])\rightarrow H^1(G_v, A_{f\chi(-j)})$ be the map on cohomology induced by $i$. For any $v\in\Sigma$, define $H^1_{\dagger}(G_v, A_{f\chi(-j)}[\pi]):=i_v^{-1}(H^1_{\dagger}(G_v, A_{f\chi(-j)}))$. Then the $\dagger$ $\pi$-Selmer group $S_{\dagger}(A_{f\chi(-j)}[\pi]/K)$ is defined as follows:
		\begin{eqnarray}\label{pi.sel.gp}
			S_{\dagger}(A_{f\chi(-j)}[\pi]/K):=\ker\Big(H^1(G_{K, \Sigma}, A_{f\chi(-j)}[\pi])\rightarrow \underset{v\in \Sigma}{\prod}\frac{H^1(G_v, A_{f\chi(-j)}[\pi])}{H^1_{\dagger}(G_v, A_{f\chi(-j)}[\pi])}\Big).
		\end{eqnarray}
      Now we   make special choices, $\dagger =BK$ and $\dagger =Gr$ respectively in  $H^1_\dagger(G_v, A_{f\chi(-j)})$ to define  the Bloch-Kato (BK) and Greenberg (Gr) $\pi$-Selmer groups as follows:

      Define $H^1_f(G_\p, V_{f\chi(-j)}):=\ker (H^1(G_\p, V_{f\chi(-j)})\rightarrow H^1(G_\p, V_{f\chi(-j)}\otimes_{\Qp}B_{crys}))$ for  $\p\mid p$, where $B_{crys}$ is the ring of periods, as defined in  \cite{fon}. Put $\dagger=BK$ and define
       \textbf{\begin{eqnarray}\label{BK.Loc.cond}
			H^1_{BK}(G_v, A_{f\chi(-j)}):=\begin{cases}
				P_v(H^1(G_v/I_v, V_{f\chi(-j)}^{I_v})), & v\in \Sigma, v\nmid p\\
				P_\p(H^1_f(G_\p, V_{f\chi(-j)})), & \p\mid p.
			\end{cases}
	\end{eqnarray}}
      \noindent Here $P_v: H^1(G_v, V_{f\chi(-j)})\rightarrow H^1(G_v, A_{f\chi(-j)})$ is  induced by the projection  $V_{f\chi(-j)}\rightarrow A_{f\chi(-j)}$.

Next assume that $f$ is a $p$-ordinary  newform and take $\dagger=Gr$ to define the Greenberg $\pi$-Selmer group. As $f$ is ordinary at $p$, we have a filtration of $V_f$ as $G_\p$-module,
	\begin{small}	\begin{eqnarray}\label{fil. of Af}
			0\rightarrow V_f^+\rightarrow V_f\rightarrow V_f^-\rightarrow0,
		\end{eqnarray}
        \end{small}
		where $V_f^-$ and $V_f^+$ are $L_\pi$-vector space of dimension $1$ and the action of $G_\p$ on $V_f^-$ is unramified. Tensoring \eqref{fil. of Af} by $\chi\omega_p^{-j}$, we have the corresponding filtration for $V_{f\chi(-j)}$. Put $A^+_{f\chi(-j)}:=P_f(V_{f\chi(-j)}^+)$ and set $A_{f\chi(-j)}^-=\frac{A_{f\chi(-j)}}{A_{f\chi(-j)}^+}$. We now define 
	\begin{small}	\begin{eqnarray}\label{Gr.Loc.cond}
			H^1_{Gr}(G_v, A_{f\chi(-j)}):=\begin{cases}
				\ker\left(H^1(G_v, A_{f\chi(-j)})\rightarrow H^1(I_v, A_{f\chi(-j)})\right), & v\in \Sigma, v\nmid p\\
				\ker\left(H^1(G_\p, A_{f\chi(-j)})\rightarrow H^1(I_\p, A_{f\chi(-j)}^-)\right), & \p\mid p.
			\end{cases}
	\end{eqnarray}   
    \end{small}
	
	%%%%%%%%%%%%%%%%%%%%%%%%%%%%%%%%%%%%%%%%%%%%%%%%%%%%%%%%%%%%%%%%%%%%%%%%%%%%%%%%%%%%%%%%%%%%%%%%%%%%%%%%%%%%%%
	\section{2-Selmer companion forms}\label{pf of main thm} We now fix $p=2$, a  number field $K$, an odd positive square-free integer $N$ and an even positive integer $k$. Let $f_1, f_2\in S_k(\Gamma_0(Np^t))$ be $p$-ordinary normalized cuspidal Hecke eigenforms  which are  newforms of level $Np^t$ with $t \geq 0$.  Let $\cK=\cK_{f_1, f_2}$ be the number field generated by the Fourier coefficients of $f_1,\, f_2$.  Let $\pi\mid p$ be a prime of $\cK$ induced by the embedding $i_p$. We fix a uniformizer, also denoted by $\pi$, of the ring of integers $\cO=\cO_{\cK_\pi}$ of the completion $\cK_\pi$. Let $\chi$ be a quadratic character of $K$ with conductor $M$ and recall $\Sigma:=\{v \text{ prime of }K: v\mid NMp\infty\}$. For  $i=1, 2$, let $S_{Gr}(A_{f_i\chi(-j)}[\pi]/K)$ be the Greenberg $\pi$-Selmer group of $f_i\otimes \chi\omega_p^{-j}$ over $K$ (see \ref{Sec 3.1}).  We recall from   \cite{JSM19}, the definition of  Greenberg $\pi$-Selmer companion  forms. 

    \begin{definition}\label{Sel.com.mod.}
		 We say $f_1$ and $f_2$ are Greenberg $\pi$-Selmer companion  forms over  $K$, if for each critical twist $j$ with $0\le j\le k-2$ and for every quadratic character $\chi$ of $K$,  there is a group isomorphism of their Greenberg $\pi$-Selmer group, i.e. 
		$S_{Gr}(A_{f_1\chi(-j)}[\pi]/K)\cong S_{Gr}(A_{f_2\chi(-j)}[\pi]/K).$     
	\end{definition}

	We now make preparation for the proof of theorem \ref{mainthm}. A prime $v\nmid p$ of $K$ is said to be `good' for a newform $f\in S_k(\Gamma_0(Np^t))$ if the Galois representation $\rho_f$ (theorem \ref{2-adic Galois rep.}) is unramified at $v$; otherwise, it is said to be `bad' at $v$. Let $\chi:G_K\rightarrow\{\pm 1\}$ be a quadratic character of $K$ and for a prime $v$ in $K$, let $\chi_v:=\chi|_{G_v}:G_v\rightarrow\{\pm1\}$ be the restriction on the decomposition group $G_v$. %Let $K_{v, \chi}:= \text{the fixed field of }\ker (\chi_v|_{I_v})$ and $$}.
      Let $I_v^\chi:=\ker (\chi_v|_{I_v})$ and  set $H_{v, \chi}:=I_v/I_v^\chi$. Note that $H_{v, \chi}\cong \ZZ/2$ if and only if $\chi$ is ramified at $v$. We say $\chi$ is `good' at a prime $v$ of $K$, if it is unramified at $v$. Otherwise, $\chi$ is said to be `bad' at $v$.  We state the following hypothesis:
	\begin{customhypo}{I}\label{I}$\cK_\pi$ is an unramified extension of $\Q_2$.
	\end{customhypo}
	\noindent Note that we can take the uniformizer $\pi$ in $\cK_\pi$ to be equal to $2$ whenever hypothesis \ref{I} holds.

    {Next, for a newform $f\in S_k(\Gamma_0(Np^t))$, we recall the definition  of conductor of $\rho_f$ and $\bar{\rho}_f$ at a finite prime $v\nmid p$, following  Serre (see \cite{li}).
\begin{definition}
     Let $f\in S_k(\Gamma_0(Np^t))$ be newform and  $\rho\in\set{\rho_f, \bar{\rho}_{f}}$. The conductor of $\rho$ at a finite prime $v\nmid p$ of $\cO_K$ is denoted by $cond_v(\rho)$ and is defined as 
    \[cond_v(\rho):=v^{n_v(\rho)},\]    
where  $n_v(\rho_f):=\dim_{\cK_\pi}\frac{V_f}{V_f^{I_v}}+sw(\rho_f^{ss}),$ and
        $n_v(\bar{\rho_f}):= \dim_{\cO/\pi}\frac{A_f[\pi]}{A_{f}[\pi]^{I_v}}+sw(\bar{\rho}_f),$ where $sw(\rho)\in \NN\cup \set{0}$ is the Swan conductor defined in \cite[\S1]{li}.     
    \end{definition}
    Further, we have that   $sw(\rho_f^{ss})=sw(\bar{\rho}_f)$ \cite[Proposition 1.1]{li}. }

    Recall $N $ is a square-free integer. Next, we state the another hypothesis:% on the conductor of $\bar{\rho}_f$.
	\begin{customhypo}{II}\label{C2}
		Let  $v\mid N$ be a prime of $K$ and $f\in S_k(\Gamma_0(Np^t))$ be a newform. Then    $$ cond_v(\bar{\rho}_{f})=cond_v({\rho}_{f}).$$
	\end{customhypo}
	\begin{remark}\label{remark 1}
	\textnormal{	If $cond_v(\bar{\rho}_f)=cond_{v}(\rho_{f})$, then $A_f^{I_v}$ is $\pi$-divisible (see \cite[Lemma 4.1.2]{epw}).}
	\end{remark} 
	Now we fix once {and} for all, $f_1, f_2\in S_k(\Gamma_0(Np^t))$ as above with $\cK=\cK_{f_1, f_2}$ and  assume that $f_1, f_2$ are in the setting of Theorem \ref{mainthm} and they satisfy all the conditions of theorem \ref{mainthm}.
Let $i_v^{ur}:H^1(I_v, A_{f\chi(-j)}[\pi])\rightarrow H^1(I_v, A_{f\chi(-j)})$ be the map on the cohomology groups induced by $i$.
	Let $\kappa_{f\chi(-j), v}:A_{f\chi(-j)}^{G_v}/\pi \rightarrow H^1(G_v, A_{f\chi(-j)}[\pi])$ and  $\kappa_{f\chi(-j), v}^{ur}:A_{f\chi(-j)}^{I_v}/\pi \rightarrow H^1(I_v, A_{f\chi(-j)}[\pi])$  denote the Kummer maps induced by the multiplication by $\pi$ on $A_{f\chi(-j)}$. Similarly, we denote by $\kappa^-_{f\chi(-j), v}$ and $\kappa_{f\chi(-j), v}^{-ur}$  the respective Kummer maps induced by the multiplication by $\pi$ on $A_{f\chi(-j)}^-$.
	From the definitions of Kummer maps, we note that $\ker i_v^{ur}=Im(\kappa_{f\chi(-j), v}^{ur})\cong A_{f\chi(-j)}^{I_v}/\pi$. Using this isomorphism, we identify $\ker i_v$ and $\ker i_v^{ur}$ with $A_{f\chi(-j)}^{G_v}/\pi$ and $A_{f\chi(-j)}^{I_v}/\pi$, respectively.  We also make a similar identification for $A_{f\chi(-j)}^-$.
	For a prime $v\nmid p$ in $K$, we have the  commutative diagram, 
	\begin{small}\begin{eqnarray}\label{Commu.diag.1}
		\begin{tikzcd}  & H^1_{Gr}(G_v, A_{f\chi(-j)}[\pi]) \arrow[r, ]\arrow[d, "i_v"]  & H^1(G_v, A_{f\chi(-j)}[\pi])\arrow[r, "\gamma_v"] \arrow[d, "i_v"] & H^1(I_v, A_{f\chi(-j)}[\pi])  \arrow[d, "i_v^{ur}"] &  \\
			0 \arrow[r,] & H^1_{Gr}(G_v, A_{f\chi(-j)}) \arrow[r, ] & H^1(G_v, A_{f\chi(-j)}) \arrow[r, "\gamma_v"] & H^1(I_v, A_{f\chi(-j)}),  
		\end{tikzcd}
	\end{eqnarray}
    \end{small}
	where $\gamma_v$ is induced by   the restriction map from $G_v$ to $I_v$. We observe that,
	\begin{eqnarray}
		\begin{aligned}
			x\in H^1_{Gr}(G_v, A_{f\chi(-j)}[\pi]) \iff \gamma_v(i_v(x))=0
			\iff i_v^{ur} (\gamma_v(x))=0\\
			\iff \gamma_v(x)\in \ker i_v^{ur} \cong A_{f\chi(-j)}^{I_v}/\pi. 
		\end{aligned}
	\end{eqnarray}
	Similarly, for $\p\mid p$, we have that $x\in H^1_{Gr}(G_\p, A_{f\chi(-j)}[\pi])\iff \gamma_\p^\prime(x)\in A_{f\chi(-j)}^{-I_\p}/\pi$, where $\gamma_\p^\prime$ is the composition of $\gamma_\p$ with the natural map $H^1(I_\p, A_{f\chi(-j)}[\pi])\rightarrow H^1(I_\p, A^-_{f\chi(-j)}[\pi])$. 
	Based on these observations, we arrive at an equivalent definition of the Greenberg $\pi$-Selmer group:
	{\small
		\begin{lemma}\label{modfi.se.gp}
			$S_{Gr}(A_{f\chi(-j)}[\pi]/K):=\ker\left(H^1(G_{K, \Sigma}, A_{f\chi(-j)}[\pi])\rightarrow \underset{v\in \Sigma v\nmid p}{\prod}\frac{H^1(I_v, A_{f\chi(-j)}[\pi])}{Im(\kappa_{f\chi(-j), v}^{ur})}\times\underset{\p\mid p}{\prod}\frac{H^1(I_\p, A_{f\chi(-j)}^-[\pi])}{Im(\kappa_{f\chi(-j), \p}^{-ur})}\right). \qed$
	\end{lemma}}

	Next, we study the Greenberg local condition at a prime $v\mid N$ of $K$. Let $f\in \{f_1, f_2\}$.  By {theorem \ref{gal-rep of f}(3)}, $A_{f}$ has the following filtration as $G_v$-module,
	\begin{small}\begin{align}\label{A'' fil}
		0\longrightarrow A_f^\prime\longrightarrow A_f\stackrel{P_v}{\longrightarrow} A_f^{\prime\prime}\rightarrow 0,  
	\end{align}\end{small}
\noindent where $A_f^\prime$ and $A_f^{\prime\prime}$ are the co-free modules of co-rank $1$ and the action of $G_v$ on $A_f^{\prime\prime}$ is via {unramified character $\eta_f$}. Taking co-invariance by $I_v$ on the $\pi$-torsion of \eqref{A'' fil}, we get: 
	$(A_f[\pi])_{I_v}\rightarrow (A_f^{\prime\prime}[\pi])_{I_v}\rightarrow 0$.
	Since  $A_f^{\prime\prime}[\pi]$ is unramified at $v$,  we get $\dim_{\cO/\pi}(A_f^{\prime\prime}[\pi])_{I_v}=\dim_{\cO/\pi}A_f^{\prime\prime}[\pi]=1$. Suppose that the hypothesis $\ref{C2}$ holds for $f$ i.e. $cond_v(\bar{\rho}_f)=cond_v(\rho_f).$ {Then},  $\dim_{\cO/\pi}(A_f[\pi])^{I_v}=1$ and consequently, $\dim_{\cO/\pi}(A_f[\pi])_{I_v}=1$ as well. Hence we arrive at the following lemma:
	\begin{lemma}\label{lemma 2.3}
		Let $v\mid N$ and $f\in \{f_1, f_2\}$. Suppose that hypothesis \ref{C2} is satisfied for $f$. Then
		\begin{eqnarray}\label{eq1}
			(A_f[\pi])_{I_v}=A_f^{\prime\prime}[\pi],
		\end{eqnarray}
		where $A_f^{\prime\prime}$ is defined in \eqref{A'' fil}.
	\end{lemma}
	
	By a slight abuse of notation, we  denote by $P_v$,  both the maps $ H^1(I_v, A_f)\rightarrow H^1(I_v, A_f^{\prime\prime})$ and $ H^1(I_v, A_f[\pi])\rightarrow H^1(I_v, A_f^{\prime\prime}[\pi])$, induced from the map $A_f\stackrel{P_v}{\rightarrow} A_f^{\prime\prime}\rightarrow 0$ in \eqref{A'' fil}. 
	\begin{lemma}\label{Lemma 2}
		Let $\chi$ be a quadratic character of $K$ and $f\in\{f_1, f_2\}$. Suppose that hypothesis \ref{C2} holds for $f$. Then for each critical twist $j$ with $0\le j\le k-2$ and for every prime $v\nmid 2\infty$, we have the following: 	
		\begin{enumerate}
			\item[$\mathrm(i)$] If both  $f$ and $\chi$ are good at $v$, then $Im(\kappa^{ur}_{f\chi(-j), v})=0.$
			\item[$\mathrm(ii)$] If $f$ is good at $v$ and $\chi$ is bad at $v$, then $Im(\kappa^{ur}_{f\chi(-j), v})= H^1(H_{v, \chi}, A_{f\chi(-j)}[\pi])$.
			\item[$\mathrm(iii)$] If $f$ is bad at $v$ and $\chi$ is good at $v$, then $Im(\kappa^{ur}_{f\chi(-j), v})=0$. 
			\item[$\mathrm(iv)$] If both $f$ and $\chi$  are bad at $v$, then $Im(\kappa^{ur}_{f\chi(-j), v})=P_v^{-1}\left(H^1(H_{v, \chi}, A_{f\chi(-j)}^{\prime\prime}[\pi])\right)$.      
		\end{enumerate}
	\end{lemma}
	$Proof.$  Note that for each critical twist $j$, a prime $v\nmid p\infty$ of $K$ is good (respectively bad) for $\chi\omega_p^{-j}$ if and only if it is good (respectively bad) for $\chi$. So without any loss of generality, we can take $j=0$ in this proof. We divide the proof into several cases discussed below: 
	\begin{itemize}
		\item [$\mathrm(i)$] Both $f$ and $\chi$  are good at $v$: In this situation $A_{f\chi}^{I_v}=A_{f\chi}$ is $\pi$-divisible, which implies  that $A_{f\chi}^{I_v}/\pi=0$ and hence $Im(\kappa^{ur}_{f\chi, v})=0$.
		
		\item [$\mathrm(ii)$] $f$ is good at $v$ and $\chi$ is bad at $v$: Consider the commutative diagram  
		{ \small	\begin{eqnarray}\label{Commu.diag.2}
				\begin{tikzcd}
					0 \arrow[r,] & A_{f\chi}^{I_v}/\pi \arrow[r, "\kappa_{f\chi, v}^{ur}"]\arrow[d, ""] & H^1(I_v, A_{f\chi}[\pi]) \arrow[r, ""]\arrow[d, ""] & H^1(I_v, A_{f\chi})\arrow[d, ""] & \\
					0\arrow[r, ]  & A_{f\chi}^{I_{v}^\chi}/\pi \arrow[r, "\kappa_{f\chi, v}^{ur}"] & H^1(I_v^\chi, A_{f\chi}[\pi])\arrow[r, ""] & H^1(I_v^\chi, A_{f\chi})
					.  
				\end{tikzcd}  
		\end{eqnarray} }
		We first note that, $A_{f\chi}^{I_v^\chi}/\pi=0$ as $f$ is good at $v$. By {the} snake lemma, we have the  exact sequence, \quad  
		$0\rightarrow A_{f\chi}^{I_v}/\pi \rightarrow H^1(H_{v, \chi}, A_{f\chi}[\pi])\rightarrow H^1(H_{v, \chi}, A_{f\chi}).$ Since $f$ is good at $v$ and $I_v$ acts on $A_{f\chi}$ non-trivially via $\chi$, {therefore $A_{f\chi}^{I_v}=A_{f\chi}[2]$. Now using hypothesis \ref{I}, we have $A_{f\chi}[2]=A_{f\chi}[\pi]$} and  hence $A_{f\chi}^{I_v}/\pi$ and $A_{f\chi}[\pi]$ have the same cardinality, given by $\dim_{\cO/\pi}A_{f\chi}[\pi]=2$. 
		We also note that $\dim_{\cO/\pi}H^1(H_{v, \chi}, \Afc[\pi])=\dim_{\cO/\pi} \Hom(H_{v, \chi}, A_{f\chi}[\pi])=2$ and hence  $\Afc^{I_v}/\pi \cong H^1(H_{v, \chi}, \Afc[\pi])$. Therefore $Im(\kappa^{ur}_{f\chi, v})=H^1(H_{v, \chi}, A_{f\chi}[\pi])$.

		\item [$\mathrm(iii)$] $f$ is bad at $v$ and $\chi$ is good at $v$: The hypothesis \ref{C2} on $f$ implies that $\Afc^{I_v}=A_f^{I_v}$ is $\pi$-divisible (see remark \ref{remark 1}) and thus  $\Afc^{I_v}/\pi=0$. Hence we get $Im(\kappa^{ur}_{f\chi, v})=0.$ 
		
		\item [$\mathrm(iv)$] Both $f$ and $\chi$  are bad at $v$: Using exact sequence \eqref{A'' fil} for $A_{f\chi}$, we get the following commutative diagram, 
		{\tiny \begin{eqnarray}\label{commu.diag.3}
				\begin{tikzcd}0\arrow[r, ]  & A_{f\chi}^{\prime I_v}[\pi] \arrow[r, ]\arrow[d, ""]  & \Afc^{I_v}[\pi]\arrow[r, ""] \arrow[d, ""] & \Afc^{\prime\prime I_v}[\pi]\arrow[r, ""]  \arrow[d, ""] & H^1(I_v, \Afc^\prime[\pi])\arrow[r, ""]\arrow[d, "\alpha_v"] &H^1(I_v, \Afc[\pi])\arrow[r, "P_v"]\arrow[d, "i_v^{ur}"] &H^1(I_v, \Afc^{\prime\prime}[\pi])\arrow[d, "\gamma_v"] \\
					0 \arrow[r,] & A_{f\chi}^{\prime I_v} \arrow[r, ] & \Afc^{I_v} \arrow[r, ""] & \Afc^{\prime\prime I_v}\arrow[r, ""] & H^1(I_v, \Afc^\prime)\arrow[r, ""]&H^1(I_v, \Afc)\arrow[r, ""]&H^1(I_v, \Afc^{\prime\prime}).  
				\end{tikzcd}
		\end{eqnarray}}
		We first note that, $I_v$ is {acts} trivially on both $A_f^\prime$ and $A_{f}^{\prime\prime}$ and $\chi$ is ramified at $v$. {Now using the hypothesis \ref{I}, we get that $A_{f\chi}^{\prime I_v}=A_{f\chi}^{\prime I_v}[\pi]$ and $A_{f\chi}^{\prime\prime I_v}=A_{f\chi}^{\prime\prime I_v}[\pi]$ and the $\cO/\pi$-ranks of both are equal to $1$.} Since $A_{f\chi}^{\prime I_v}=A_{f\chi}^{\prime I_v}[\pi]$ and $A_{f\chi}^{\prime\prime I_v}=A_{f\chi}^{\prime\prime I_v}[\pi]$ are finite, it follows that $A_{f\chi}^{I_v}$ is also finite and hence $\dim_{\cO/\pi}A_{f\chi}^{I_v}/\pi=\dim_{\cO/\pi}A_{f\chi}^{I_v}[\pi]$. %{By hypothesis \ref{C2} on $f$, the $\ast$ appearing in theorem \ref{2-adic Galois rep.}(3) satisfies $\ast(\sigma)\ne  0\pmod{\pi}$, for some $\sigma\in G_v$.}
       {The hypothesis \ref{C2} on $f$ implies that the action of $I_v$ on $A_{f\chi}[\pi]$ is non-trivial. Using this observation together with the fact that $A^\prime_{f\chi}[\pi]\hookrightarrow A_{f\chi}^{I_v}[\pi]$, it follows that $A_{f\chi}^{I_v}[\pi]$ has $\cO/\pi$-rank $1$ and thus $\dim_{\cO/\pi}A_{f\chi}^{I_v}/\pi=1$}, which in turn shows that $\Afc^{\prime I_v}[\pi]\rightarrow\Afc^{I_v}[\pi]$ is an isomorphism. %Further, from the commutative diagram \eqref{commu.diag.3}, we derive  the following commutative diagram, 
	%	{\footnotesize
			%\begin{eqnarray}\label{commu.diag.4}
		%		\begin{tikzcd}0\arrow[r, ]  & \Afc^{\prime\prime I_v}[\pi]\arrow[r, ""]  \arrow[d, ""] & H^1(I_v, \Afc^\prime[\pi])\arrow[r, ""]\arrow[d, "\alpha_v"] &H^1(I_v, \Afc[\pi])\arrow[r, "P_v"]\arrow[d, "i_v^{ur}"] &H^1(I_v, \Afc^{\prime\prime}[\pi])\arrow[d, "\gamma_v "] \\
	%				& A_{f\chi}^{\prime\prime I_v}\arrow[r, ""]  & H^1(I_v, \Afc^\prime)[\pi]\arrow[r, ""]&H^1(I_v, \Afc)[\pi]\arrow[r, "P_v"]&H^1(I_v, \Afc^{\prime\prime})[\pi].  
	%			\end{tikzcd}
	%		\end{eqnarray}
	%	} 
		We observe that $H^1(I_v, A_{f\chi}^{\prime}[\pi])=\Hom(I_v, A_{f\chi}^{\prime}[\pi])$. It is well known that there is only one quadratic extension of $K_v^{ur}$ as $v\nmid 2$, therefore $\dim_{\cO/\pi}\Hom(I_v, A_{f\chi}^{\prime}[\pi])=1$. This gives us $A_{f\chi}^{\prime\prime}[\pi]\cong H^1(I_v, A_{f\chi}^\prime[\pi])$ and hence the map $P_v$ is injective. By {the} snake lemma, we get that $\ker i_v^{ur}\hookrightarrow \ker\gamma_v$. Since $\dim_{\cO/\pi}\ker i_v^{ur}=\dim_{\cO/\pi}\ker \gamma_v=1$, we get an isomorphism    from   $\ker i_v^{ur}$ to  $\ker\gamma_v$ induced by $P_v$. Consider a commutative diagram, similar to \eqref{Commu.diag.2}, for  $A_{f\chi}^{\prime\prime}$,
		
		{\small 	\[
			\begin{tikzcd}
				0 \arrow[r,] & A_{f\chi}^{\prime\prime I_v}/\pi \arrow[r, ]\arrow[d, ""] & H^1(I_v, A^{\prime\prime}_{f\chi}[\pi]) \arrow[r, "\gamma_v"]\arrow[d, ""] & H^1(I_v, A^{\prime\prime}_{f\chi})[\pi]\arrow[d, ""] & \\
				0\arrow[r, ]  & A_{f\chi}^{\prime\prime I_{v}^\chi}/\pi \arrow[r, ] & H^1(I_v^\chi, A^{\prime\prime}_{f\chi}[\pi])\arrow[r, ""] & H^1(I_v^\chi, A^{\prime\prime}_{f\chi})[\pi]
				.  
			\end{tikzcd}
			\] }
		Since $I_v^{\chi}$ acts trivially on $A_{f\chi}^{\prime\prime}$ and $A_{f\chi}^{\prime\prime}$ is divisible,  $A_{f\chi}^{\prime\prime I_v^\chi}/\pi=0$. By {the} snake lemma we have, $\Afc^{\prime\prime I_v}/\pi \hookrightarrow H^1(H_{v, \chi}, \Afc^{\prime\prime}[\pi])$ and {theses two groups} have the same size, hence
		$\Afc^{\prime\prime I_v}/\pi\cong H^1(H_{v, \chi}, \Afc^{\prime\prime}[\pi]). $
		Now from the commutative diagram \eqref{commu.diag.3}, we deduce that 
		$$Im(\kappa_{f\chi, v}^{ur})=P_v^{-1}(\ker \gamma_v)=P_v^{-1}(H^1(H_{v, \chi}, \Afc^{\prime\prime}[\pi])).\qed$$
	\end{itemize}
	
	Next we analyse the local conditions at a prime $\p$  of $K$ dividing $p$. Let $f\in\{f_1, f_2\}$ and recall that  $\omega_p:G_\p\rightarrow \Zp^\ast$ is the $p$-adic cyclotomic character. %As $k$ is even, note that $\omega_p^{k-1}\ne 1\pmod{\pi^2}$.
    For any quadratic character $\chi$ of $K$ and for each critical twist $j$ with $0\le j\le k-2$, via $p$-ordinarity of $f$, we have the following filtration of $A_{f\chi(-j)}:=\frac{V_f\otimes\chi\omega_p^{-j}}{T_f\otimes\chi\omega_p^{-j}}$ as $G_\p$-module (see \S\ref{Sec 3.1}), 
	\[0\rightarrow A_{f\chi(-j)}^+\rightarrow A_{f\chi(-j)}\rightarrow A_{f\chi(-j)}^-\rightarrow 0.\]
	Also recall that $\kappa_{f\chi(-j),\p}^{-ur}:A_{f\chi(-j)}^{-I_\p}/\pi\rightarrow H^1(I_\p, A^-_{f\chi(-j)}[\pi])$ is the Kummer map induced from the multiplication by $\pi$ onto $A_{f\chi(-j)}^-$.

	\begin{lemma}\label{Lemma 4}
		Let $f_1,f_2\in S_k(\Gamma_0(Np^t))$ be as in theorem \ref{mainthm}. Assume that $\omega_{p|_{G_p}}\not\equiv 1\pmod{\pi^2}$. Let $\phi:A_{f_1}[\pi^2]\rightarrow A_{f_2}[\pi^2]$ be a $G_\p$-isomorphism, then $\phi$ maps $A^-_{f_1}[\pi]$ onto $A^-_{f_2}[\pi]$.  
		
	\end{lemma}
	\begin{proof}
	{Since $f_1$ and $f_2$ are ordinary at $p$, there {exists} an $\cO$-basis $\{e_1, e_2\}$ (respectively \{$v_1, v_2$\}) of $T_{f_1}$(respectively $T_{f_2}$) such that, with respect to these bases, $\rho_{f_i}$ restricted to $G_\p$ has the following form (see theorem \ref{gal-rep of f}(2)), 
		\begin{small}\begin{eqnarray*}
			\rho_{f_i}|_{G_\p}\sim \begin{pmatrix}
				\delta_{f_i}^{-1}\omega_p^{k-1} & \ast \\
				0 & \delta_{f_i}
			 \end{pmatrix},\,\,\,\,\,\,
            \text{for } i=1,2.
		\end{eqnarray*}
        \end{small} }
		Recall that $A_{f_i}=V_{f_i}/T_{f_i}$. Now $A_{f_i}[\pi^2]$ and $A_{f_i}^+[\pi^2]$ can be represented as follows:
		\[
		A_{f_1}[\pi^2]:=\left\{\frac{x_1}{\pi^2}e_1+\frac{y_1}{\pi^2}e_2+T_{f_1}:x_1, y_1\in \cO \right\}, 
		A_{f_2}[\pi^2]:=\left\{\frac{a_1}{\pi^2}v_1+\frac{b_1}{\pi^2}v_2+T_{f_2}:a_1, b_1\in \cO\right\}\]
		\[A^+_{f_1}[\pi^2]:=\left\{\frac{x_1}{\pi^2}e_1+T_{f_1}:x_1\in \cO\right\}, 
		A^+_{f_2}[\pi^2]:=\left\{\frac{a_1}{\pi^2}v_1+T_{f_2}:a_1\in \cO \right\}.\]
		
		We claim that  $\phi$ induces an isomorphism from $A^-_{f_1}[\pi]$ to $A^-_{f_1}[\pi]$. To show this, it is enough to show that $\phi$ maps $A^+_{f_1}[\pi]$ to $A^+_{f_2}[\pi]$. The action of $\sigma\in G_\p$ on $A_{f_1}[\pi^2]$ is as follows:
		\begin{eqnarray*}
			\sigma\cdot \left(\frac{x_1}{\pi^2}e_1+\frac{y_1}{\pi^2}e_2+T_{f_1}\right)=\frac{\delta_{f}^{-1}(\sigma)\omega_p^{k-1}(\sigma)x_1+\ast(\sigma)y_1}{\pi^2}e_1+\frac{\delta_{f}(\sigma)y_1}{\pi^2}e_2+T_{f_1}.
		\end{eqnarray*}
		Let $\alpha\in A^+_{f_1}[\pi]$, then there exist $\beta\in A_{f_1}^+[\pi^2]$ such that $\alpha=\pi\beta$. Let $\phi(\beta)=\frac{a_1}{\pi^2}v_1+\frac{b_1}{\pi^2}v_2+T_{f_2}\in A_{f_2}[\pi^2]$ for some $a_1, b_1\in \cO$, then for any $\sigma\in G_\p$ 
		\begin{eqnarray}\label{eq2}
			\sigma\cdot\phi(\beta)=\frac{\delta_{f}^{-1}(\sigma)\omega_p^{k-1}(\sigma)a_1+\ast(\sigma)b_1}{\pi^2}v_1+\frac{\delta_{f}(\sigma)b_1}{\pi^2}v_2+T_{f_2}.
		\end{eqnarray}
		On the other hand,  as $\phi$ is a $G_\p$-isomorphism and $\beta \in A_{f_1}^+[\pi^2]$, it follows that for every $\sigma \in G_\p$,
		\begin{eqnarray}\label{eq3}
			\sigma\cdot\phi(\beta)=\phi(\sigma\cdot \beta)=\delta_{f}^{-1}(\sigma)\omega_p^{k-1}(\sigma)\phi(\beta).
		\end{eqnarray}
		From \eqref{eq2} and \eqref{eq3}, we deduce that for all $\sigma\in I_\p$,
		{\footnotesize
			\begin{eqnarray}\label{eq5}
				\frac{\delta_{f}^{-1}(\sigma)\omega_p^{k-1}(\sigma)a_1+\ast(\sigma)b_1}{\pi^2}v_1+\frac{\delta_{f}(\sigma)b_1}{\pi^2}v_2+T_{f_2} =  \frac{\delta_{f}^{-1}(\sigma)\omega_p^{k-1}(\sigma)a_1}{\pi^2}v_1+ \frac{\delta_{f}^{-1}(\sigma)\omega_p^{k-1}(\sigma)b_1}{\pi^2}v_2+T_{f_2}. 
		\end{eqnarray}} 
		Now \eqref{eq5} implies that 
		$(\delta_{f}^{-1}(\sigma)\omega_p^{k-1}(\sigma)-\delta_{f}(\sigma))b_1/\pi^2\in \cO_{f_1}$. { By our assumption that $\omega_{p|_{G_p}}\not\equiv 1\pmod{\pi^2}$ and $\omega_p$ is ramified at $\p$, it follows that  $\omega_p^{k-1}(\sigma)=-1$ for some $\sigma\in I_\p$}. By using hypothesis \ref{I}, we deduce that $\pi\mid b_1$. Further using  $\phi(\alpha)=\pi\phi(\beta)$, we get $\phi(\alpha)=\frac{\pi a_1}{\pi^2}v_1+\frac{\pi b_1}{\pi^2}v_2+T_{f_2}.$
		Hence $\phi(\alpha)=\frac{a_1}{\pi}v_1+T_{f_2}\in A^+_{f_2}[\pi].$ This shows that given $\alpha\in A_{f_1}^+[\pi]$,  $\phi(\alpha)\in A_{f_2}^+[\pi]$ and this completes the proof.
	\end{proof}
\begin{remark}
   \textnormal{ The result corresponding to lemma \ref{Lemma 4} for an odd prime is discussed in \cite[Lemma 3.11]{JSM19}. However the crucial assumption in \cite[Lemma 3.11]{JSM19} is that $\omega_p^{k-1}\epsilon_{i, p}\ne 1\pmod{\pi}$, $\epsilon_{i,p}$ is the $p$-part of the nebentypus $\epsilon_i$ of $f_i\in S_k(\Gamma_0(Np^t), \epsilon_i)$. As indicated in the proof of \cite[Lemma 3.11]{JSM19}, this is to ensure that the action of $I_p$ on $A_{f_i}^+$ is non-trivial, so that $H^0(I_p, A_{f_i}^+)=0$.}

 \textnormal{   Note that in our setting, $p=2$, $k$ is even and nebentypus is trivial, so $\epsilon_{i, p}$ is trivial, thus we necessarily  have $\omega_p^{k-1}\equiv 1\pmod{\pi}$ in lemma \ref{Lemma 4}.  The proof of \cite[Lemma 3.11]{JSM19} does not hold and a different strategy is used in the proof lemma \ref{Lemma 4}. In fact we use a higher congruence i.e. a $G_\p$-isomorphism between $A_{f_1}[\pi^2]\cong  A_{f_2}[\pi^2]$ to identify $A^-_{f_1}[\pi]$ with $A^-_{f_2}[\pi]$ in lemma \ref{Lemma 4}.  Note that at $p=2$, Mazur-Rubin assumes $\EC_1[4]\cong \EC_2[4]$ for the corresponding result in \cite[Theorem 3.1]{mr}. Thus our hypothesis in lemma \ref{Lemma 4}  is compatible with the corresponding hypothesis of \cite{mr}.  }
\end{remark} 
    
	\begin{lemma}\label{Lemma 3}
		Let $f_1,f_2\in S_k(\Gamma_0(Np^t))$ be $p$-ordinary newforms and  $\omega_{p|_{G_p}}\not\equiv 1\pmod{\pi^2}$. Let $\phi:A_{f_1}[\pi^2]\rightarrow A_{f_2}[\pi^2]$ be an $G_\p$-isomorphism,  then for every quadratic character $\chi$ of $K$ and for every $j$ with $0\le j\le k-2$, we have 
		$$\phi(Im(\kappa^{-ur}_{f_1\chi(-j), \p}))=Im(\kappa^{-ur}_{f_2\chi(-j), \p}).$$
	\end{lemma}
	\begin{proof}
		Note that both the $G_\p$ modules $A^-_{f_1}$ and $A^-_{f_2}$ are  unramified. Fix an $I_\p$-isomorphism $\psi:A^-_{f_1}\rightarrow A^-_{f_2}$; this induces an $I_\p$-isomorphism  $\psi:A^-_{f_1\chi(-j)}\rightarrow A^-_{f_2\chi(-j)}$. Let $\psi_{|_\pi}:A^-_{f_1\chi(-j)}[\pi]\rightarrow A^-_{f_2\chi(-j)}[\pi]$ be the restriction of the map $\psi$. Now the $G_\p$-isomorphism $\phi:A_{f_1}[\pi^2]\rightarrow A_{f_2}[\pi^2]$ induces the map $\phi=\phi|_{A_{f_1}[\pi]}: A_{f_1}[\pi]\rightarrow A_{f_2}[\pi]$. By Lemma \ref{Lemma 4}, $\phi$ induces a $G_\p$-isomorphism from $ A^-_{f_1}[\pi]$ to $ A^-_{f_2}[\pi]$. Further  taking twist by $\chi\omega_p^{-j}$, we get an $I_\p$-isomorphism from $ A^-_{f_1\chi(-j)}[\pi]$ to $ A^-_{f_2\chi(-j)}[\pi]$ which we continue to denote by $\phi$. Since $\dim_{\cO/\pi}A^-_{f_1\chi(-j)}[\pi]=1=\dim_{\cO/\pi}A^-_{f_2\chi(-j)}[\pi]$,  therefore there exists $\alpha\in (\cO/\pi)^\ast$ such that $\psi_{|_\pi}=\alpha \phi$. Let $\tilde{\psi}_{|_\pi}\text{ and } \tilde{\phi}: H^1(I_\p, A^-_{f_1\chi(-j)}[\pi])\rightarrow  H^1(I_\p, A^-_{f_2\chi(-j)}[\pi])$ be the maps induced from $\psi_{|_\pi}$ and $\phi$, respectively. 
		Now consider the commutative diagram,
		{\small \begin{eqnarray*}
				\begin{tikzcd}
					0\arrow[r, ""] & A_{f_1\chi(-j)}^{-I_\p}/\pi\arrow[r, "\kappa_{f_1\chi(-j), \p}^{-ur}"]\arrow[d, ""] &  H^1(I_\p, A^-_{f_1\chi(-j)}[\pi])\arrow[r, "i_1^{ur}"]\arrow[d, "\tilde{\psi}_{|_\pi}=\alpha \tilde{\phi}"] & H^1(I_\p, A^-_{f_1\chi(-j)})\arrow[d, "\tilde{\psi} "] \\
					0\arrow[r, ""] & A_{f_2\chi(-j)}^{-I_\p}/\pi\arrow[r, "\kappa_{f_2\chi(-j), \p}^{-ur}"] & H^1(I_\p, A^-_{f_2\chi(-j)}[\pi])\arrow[r, "i_2^{ur}"] & H^1(I_\p, A^-_{f_2\chi(-j)}). 
				\end{tikzcd}
		\end{eqnarray*}}
		Since both maps $\tilde{\psi}_{|_\pi}$ and $\tilde{\psi}$ are isomorphism, it follows that $ \tilde{\psi}_{|_\pi}(\ker i_1^{ur})= \ker i_2^{ur}.$ Consequently, $\alpha\tilde{\phi}(Im(\kappa^{-ur}_{f_1\chi(-j), \p}))=Im(\kappa^{-ur}_{f_2\chi(-j), \p})$ and hence $\phi(Im(\kappa^{-ur}_{f_1\chi(-j), \p}))=Im(\kappa^{-ur}_{f_2\chi(-j), \p})$.
	\end{proof}
	Finally, we discuss the situation  when $v\mid \infty$. Since only a real place can ramify, we can take $v$ to be a real place of $K$. In this situation $G_v=I_v=\{1, c\}$, where $c$ denotes the complex conjugation. Let $f\in \{f_1, f_2\}$. Recall that $\rho_f$ is an odd Galois representation i.e. $\det(\rho_f(c))=-1$. Therefore, $\rho_{f_{|G_v}}\sim\begin{pmatrix}
		1 & 0 \\
		0 & -1
	\end{pmatrix},$
	which induces a split exact sequence:   
    \begin{small}
	\begin{eqnarray}\label{split exact}
		\begin{tikzcd}
0 \arrow[r] & A_f^\prime \arrow[r, "J^\prime_v"] & A_f \arrow[r, ""]  & A_{f}^{\prime\prime} \arrow[r]\arrow[l, bend left, "J_v^{\prime\prime}"] & 0.
\end{tikzcd}
\end{eqnarray}
    \end{small}
	 Note that the action of $G_v$ is trivial on $A_{f}^{'}$ and non-trivial on $A_f^{''}$. %Let $J_v^{\prime\prime}:H^1(I_v, A_f^{\prime\prime}[\pi])\rightarrow H^1(I_v, A_f[\pi])$ (respectively  $J_v^\prime:H^1(I_v, A_f^\prime[\pi])\rightarrow H^1(I_v, A_f[\pi])$) be the maps on cohomology groups induced from the inclusion $A_{f}^\prime[\pi]\rightarrow A_f[\pi]$ (respectively $A_{f}^{\prime\prime}[\pi]\rightarrow A_f[\pi]$).   
	\begin{lemma}\label{Lemma 5}
		Let $v\mid\infty$ be a prime of $K$ and $f\in\{f_1, f_2\}$. Then for every quadratic character $\chi$ of $K$ and for every $j$ with $0\le j\le k-2$, we have the following: 
		\begin{itemize}
			\item If $\chi\omega_p^{-j}$ is unramified at $v$, then
			$Im(\kappa^{ur}_{f\chi(-j), v})= J_v^{\prime\prime}(H^1(I_v, A_{f\chi(-j)}^{\prime\prime}[\pi]))$.
			\item If $\chi\omega_p^{-j}$ is ramified at $v$, then
			$Im(\kappa_{f\chi(-j), v}^{ur})=J_v^\prime(H^1(I_v, A_{f\chi(-j)}^{\prime}[\pi]))$.
		\end{itemize}
	\end{lemma}
	\begin{proof}
		We have $G_v=I_v\cong \ZZ/2\ZZ$. Tensoring the  exact sequence \eqref{split exact} by $\chi\omega_p^{-j}$ and % is split, it follows that $0\rightarrow A_{f\chi(-j)}^{\prime\prime}\stackrel{}{\rightarrow} A_{f\chi(-j)}\stackrel{}{\rightarrow} A_{f\chi(-j)}^\prime\rightarrow0$ is also exact.  
        consider the following commutative diagram, 
		\begin{small}\begin{eqnarray}\label{Commu.diag.6}
			\begin{tikzcd}0\arrow[r, ]  & H^1(I_v, A_{f\chi(-j)}^{\prime}[\pi]) \arrow[r, "J^{\prime}_v"]\arrow[d, "\gamma_v"]  & H^1(I_v, A_{f\chi(-j)}[\pi])\arrow[r, ""] \arrow[d, "i_v^{ur}"] & H^1(I_v, A_{f\chi(-j)}^{\prime\prime}[\pi])\arrow[r] \arrow[d, "\theta_v"] \arrow[l, bend left, "J_v^{\prime\prime}"]& 0 \\
				0 \arrow[r,] & H^1(I_v, A_{f\chi(-j)}^{\prime}) \arrow[r, ] & H^1(I_v, A_{f\chi(-j)}) \arrow[r, ""] & H^1(I_v, A_{f\chi(-j)}^{\prime\prime})\arrow[r] &0.  
			\end{tikzcd}
		\end{eqnarray}
        \end{small}
	First, we consider the case when $\chi\omega_\p^{-j} $ is unramified at $v$.	Since $A_{f\chi(-j)}^{'I_v}=A_{f}^{'}$ is $\pi$-divisible, it follows that $\ker\gamma_v \cong A_{f\chi(-j)}^{'I_v}/\pi=0$ and hence by {the} snake lemma, we get that $\ker\theta_v\hookrightarrow\ker i_v^{ur}.$
		Since $I_v\cong \ZZ/2\ZZ$ is {acts} trivially on $A_{f\chi(-j)}^{\prime\prime}[\pi]$, we deduce that $\dim_{\cO/\pi}H^1(I_v, A_{f\chi(-j)}^{\prime\prime}[\pi])=1$. We also note that $\dim_{\cO/\pi}\ker \theta_v=\dim_{\cO/\pi} A_{{f\chi(-j)}}^{\prime\prime I_v}/\pi=\dim_{\cO/\pi}A_{f\chi(-j)}^{\prime\prime}[\pi]=1$ and thus  $\ker\theta_v=H^1(I_v, A_{f\chi(-j)}^{\prime\prime}[\pi]).$ This shows that, 
		\[Im(\kappa^{ur}_{f\chi(-j), v})=\ker i_v^{ur}= J_v^{\prime\prime}(H^1(I_v, A_{f\chi(-j)}^{\prime\prime}[\pi])).\]
		
		Now consider the second case i.e., $\chi\omega_p^{-j}$ is ramified at $v$, then the action of $I_v$ is non-trivial on $A_{f\chi(-j)}^{\prime}$ and trivial on $A_{f\chi(-j)}^{\prime\prime}$. Therefore $A_{f\chi(-j)}^{\prime\prime}$ is $\pi$-divisible and so $\ker\theta_v=0$. 
        %Tensoring the exact sequence \eqref{split exact} by $\chi\omega_p^{-j}$, we obtain the exact sequence, $0\rightarrow A_{f\chi(-j)}^{\prime}\stackrel{}{\rightarrow} A_{\chi(-j)}\stackrel{}{\rightarrow} A_{f\chi(-j)}^{\prime\prime}\rightarrow0.$
        Using a similar argument as in the previous case, we get $Im(\kappa_{f\chi(-j), v}^{ur})=\ker i_v^{ur}=J_v^\prime(H^1(I_v, A_{f\chi(-j)}^{\prime}[\pi])).$
	\end{proof}
	We continue our study at $v\mid \infty$. 
	\begin{lemma}\label{Lemma 4.6}
		Let $f_1,f_2\in S_k(\Gamma_0(Np^t))$ be as in theorem \ref{mainthm} and  $\phi:A_{f_1}[\pi^2]\rightarrow A_{f_2}[\pi^2]$ be a $G_v$-isomorphism, where $v\mid \infty$. Then $\phi$ induces a $G_v$-isomorphism from $A_{f_1}^\prime[\pi]$ to $A_{f_2}^\prime[\pi]$ and from $A_{f_1}^{\prime\prime}[\pi]$ to $A_{f_2}^{\prime\prime}[\pi]$ respectively. 
	\end{lemma}
	\begin{proof}
		Recall that for $i=1, 2$,  $\rho_{f_i|_{G_v}}\sim\begin{pmatrix}
			1 & 0 \\
			0 & -1
		\end{pmatrix}.$ We replace the matrix used in the proof of the lemma \ref{Lemma 4} with    $\rho_{f_{|G_v}}\sim\begin{pmatrix}
			1 & 0 \\
			0 & -1
		\end{pmatrix}$ and then a proof similar to lemma \ref{Lemma 4} works in this case.  
	\end{proof}
	Now we are ready to complete the proof of theorem \ref{mainthm}.

	\subsection{Proof of Theorem \ref{mainthm}} Recall $p=2$, $f_1, f_2$ are in the setting of theorem \ref{mainthm} and they satisfy hypothesis \ref{I}, hypothesis \ref{C2}.
	
	\noindent\textbf{Proof of part $(i)$:} In this case, we  have a given $G_K$-isomorphism $\phi:A_{f_1}[\pi^2]\rightarrow A_{f_2}[\pi^2]$. In particular, $\phi$ induces a $G_K$-isomorphism $\phi:A_{f_1}[\pi]\rightarrow A_{f_2}[\pi]$. Let $\chi$ be a quadratic character of $K$ and put $M=cond(\chi)$. Recall that $\Sigma$ is a finite set of primes of $K$ given by  $\Sigma=\{v \text{ prime in }K: v\mid NMp\infty\}$. Note that for every quadratic character $\chi$ of $K$ and for every $j$ with $0\le j\le k-2$, the isomorphism $\phi$ induces an isomorphism $\tilde{\phi}:H^1(G_{K, \Sigma}, A_{f_1\chi(-j)}[\pi])\rightarrow H^1(G_{K, \Sigma}, A_{f_2\chi(-j)}[\pi])$. We will show that $\tilde{\phi}$ induces an isomorphism  $S_{Gr}(A_{f_1\chi(-j)}[\pi]/K))\cong S_{Gr}(A_{f_2\chi(-j)}[\pi]/K)$ by proving that the local condition of $S_{Gr}(A_{f_1\chi(-j)}[\pi]/K)$ appearing in \eqref{Gr.Loc.cond} maps to the corresponding local condition of $S_{Gr}(A_{f_2\chi(-j)}[\pi]/K)$ for all $v\in \Sigma$. Now in view of lemma \ref{modfi.se.gp}, to prove the theorem,  it suffices to establish the following two results: (i) For every prime $v\in \Sigma, v\nmid p$ under the isomorphism $\tilde{\phi}: H^1(I_v, A_{f_1\chi(-j)}[\pi])\rightarrow H^1(I_v, A_{f_2\chi(-j)}[\pi])$, the sub-module $Im(\kappa^{ur}_{f_1\chi(-j), v})$ maps to $Im(\kappa^{ur}_{f_2\chi(-j), v})$. (ii) And    in the view of lemma \ref{modfi.se.gp} and lemma \ref{Lemma 3},  for every prime $\p\mid p$ of $K$, under the isomorphism $\tilde{\phi}:H^1(I_\p, A^-_{f_1\chi(-j)}[\pi])\rightarrow H^1(I_\p, A^-_{f_2\chi(-j)}[\pi])$, the sub-module $Im(\kappa^{-ur}_{f_1\chi(-j), \p})$ maps to $Im(\kappa^{-ur}_{f_2\chi(-j), \p})$. These proofs are subdivided into three parts; 
	\begin{enumerate}
		\item $v\nmid p\infty$: 
		\begin{enumerate}
			\item[$\mathrm(a)$]  If $v\nmid N$, then $v$ is good for both $f_1, f_2$. In this situation, lemma \ref{Lemma 2}$\mathrm(i)$, $\mathrm(ii)$ implies that $\tilde{\phi}(Im(\kappa^{ur}_{f_1\chi(-j), v}))=Im(\kappa^{ur}_{f_2\chi(-j), v})$. 
			\item[$\mathrm(b)$] If $v\mid N$, then lemma \ref{Lemma 2}$\mathrm(iii)$ implies  that  
			$\tilde{\phi}(Im(\kappa^{ur}_{f_1\chi(-j), v}))=Im(\kappa^{ur}_{f_2\chi(-j), v})=0$, when $\chi$ is unramified at $v$. On the other hand, if $\chi$ is ramified at $v$ then by lemma \ref{Lemma 2}$\mathrm(iv)$, we get that $Im(\kappa^{ur}_{f_i\chi(-j), v})=P_v^{-1}(H^1(H_{v, \chi}, A_{f_i\chi(-j)}^{\prime\prime}[\pi]) \text{ for }  $i=1, 2$.$
			Further by lemma \ref{lemma 2.3}, we identify $A_{f_i}^{\prime\prime}[\pi]=(A_{f_i}[\pi])_{I_v}$. Thus $\phi$ maps $A_{f_1}^{\prime\prime}[\pi]$ onto $A_{f_2}^{\prime\prime}[\pi]$ and in turn $\phi$ maps $A_{f_1\chi(-j)}^{\prime\prime}[\pi]$ onto $A_{f_2\chi(-j)}^{\prime\prime}[\pi]$. Hence 
			$\tilde{\phi}(Im(\kappa^{ur}_{f_1\chi(-j), v}))=Im(\kappa^{ur}_{f_2\chi(-j), v}).$
		\end{enumerate}
		\item $\p\mid 2$:
		The proof follows directly from lemma \ref{Lemma 3}. 
		\item $v\mid \infty$:  In this case {the} proof follows from lemma \ref{Lemma 5} and lemma \ref{Lemma 4.6}.\qed 
	\end{enumerate}    
	
	\noindent\textbf{Proof of Part $(ii)$:-}
	In this case, we have a $G_K$-isomorphism $\phi:A_{f_1}[\pi]\rightarrow A_{f_2}[\pi]$ which again induces an isomorphism $\tilde{\phi}:H^1(G_{K, \Sigma}, A_{f_1\chi(-j)}[\pi])\rightarrow H^1(G_{K, \Sigma}, A_{f_2\chi(-j)}[\pi])$. {Further, $K$ has no real place and $\bar{\rho}_{f_1}$ (and hence $\bar{\rho}_{f_2}$) is ramified at every prime $\p\mid p$ of $K$}.
	
	If $v\in \Sigma, v\nmid p$, then as in the proof of part $(i)$, under the isomorphism $\tilde{\phi}: H^1(I_v, A_{f_1\chi(-j)}[\pi])\rightarrow H^1(I_v, A_{f_2\chi(-j)}[\pi])$, the sub-module $Im(\kappa^{ur}_{f_1\chi(-j), v})$ maps to $Im(\kappa^{ur}_{f_2\chi(-j), v})$ by lemma \ref{Lemma 2}.   Next, as $K$ has no real embedding, local conditions at infinite places do not contribute in the $\pi$-Selmer group. 
	
	For a prime $\p\mid p$, we claim that:  the $G_K$-isomorphism $\phi:A_{f_1\chi(-j)}[\pi]\rightarrow A_{f_2\chi(-j)}[\pi]$ maps $A_{f_1\chi(-j)}^-[\pi]$ to $A_{f_2\chi(-j)}^-[\pi]$, for every $j$. Assuming the claim, following the proof of lemma \ref{Lemma 3}, we see that under the isomorphism $\tilde{\phi}:H^1(I_\p, A^-_{f_1\chi(-j)}[\pi])\rightarrow H^1(I_\p, A^-_{f_2\chi(-j)}[\pi])$, the sub-module $Im(\kappa^{-ur}_{f_1\chi(-j), \p})$ maps to $Im(\kappa^{-ur}_{f_2\chi(-j), \p})$. To establish the claim, observe that  under the assumption that $\bar{\rho}_{f_i}$ is ramified at $\p$, we have $\dim_{\cO/\pi}\, H^0(G_\p, A_{f_i\chi(-j)}[\pi]) \leq 1$, for $i=1, 2$. On the other hand, $A_{f_i\chi(-j)}^+[\pi]\hookrightarrow H^0(G_\p, A_{f_i\chi(-j)}[\pi])$. It follows that $A_{f_i\chi(-j)}^+[\pi]= H^0(G_\p, A_{f_i\chi(-j)}[\pi])$, for $i=1, 2$.  Therefore, $\phi(A_{f_1\chi(-j)}^+[\pi])=A_{f_2\chi(-j)}^+[\pi]$. Moreover, for $i=1, 2$, we have $A_{f_i\chi(-j)}^-[\pi]=\frac{A_{f_i\chi(-j)}[\pi]}{A_{f_i\chi(-j)}^+[\pi]}$ and this shows that $\phi(A_{f_1\chi(-j)}^-[\pi])=A_{f_2\chi(-j)}^-[\pi]$, as required.\qed

	\subsection{Example}\label{examp.}
	Consider the elliptic curves $\EC_1$ and $\EC_2$ of LMFDB label \href{https://www.lmfdb.org/EllipticCurve/Q/158/a/1}{158.a1} and \href{https://www.lmfdb.org/EllipticCurve/Q/158/b/1}{158.b1} respectively and fix $p=2$.  Let $f_1, f_2\in S_2(\Gamma_0(158))$ be the corresponding newforms obtained via modularity. Using sturm bound via SAGE, we show that $a_q(f_1)\equiv a_q(f_2)\pmod{2}$ for every rational prime $q$. The Fourier expansions of $f_1$ and $f_2$ are given by, 
	\[f_1(q)=q - q^2 - q^3 + q^4 - q^5 + q^6 - 3q^7 - q^8 - 2q^9 + O(q^{10}),\]
	\[f_2(q)= q - q^2 + q^3 + q^4 + 3q^5 - q^6 - q^7 - q^8 - 2q^9 + O(q^{10}).
	\]
     { Since $f_1-f_2\in S_2(\Gamma_0(158))$, by  \cite[Theorem 9.18]{stein}, if $a_q(f_1)\equiv a_q(f_2)\pmod{2}$ for all prime $q\le   
	\frac{k[\SL_2(\ZZ):\Gamma_0(N)]}{12}-\frac{[\SL_2(\ZZ):\Gamma_0(N)]-1}{N}$, then $a_q(f_1)\equiv a_q(f_2)\pmod{2}$ for every rational prime $q$.   
In this case,	we compute the Sturm bound $\frac{k[\SL_2(\ZZ):\Gamma_0(N)]}{12}-\frac{[\SL_2(\ZZ):\Gamma_0(N)]-1}{N}\approx 38.5$. %Thus we need to check for primes $p\le 37.5$.  
	  By SAGE, we check that  $a_q(f_1)\equiv a_q(f_2)\pmod{2}$ for all prime $q\le 37$ and for $q=2, 79$. Thus $a_q(f_1)\equiv a_q(f_2)\pmod{2}$ for all prime $q$.}
	Note that the level $158=2\times 79$ is square-free and the coefficient field  $\cK_{f_1, f_2}=\Q$, therefore hypothesis \ref{I} holds. {Put $K:=\Q(\sqrt{-3})$, then $79$ is unramified in $K$. Using \cite[Proposition 2.12(c)]{DDT}, we deduce that for any prime $v\mid79$ of $K$ and $i=1,2$,  $cond_v (\rho_{f_i})=cond_v(\bar{\rho}_{f_i})=79$ i.e. $f_1$ and $f_2$ satisfy hypothesis \ref{C2}}. Using LMFDB, we deduce that  $\bar{\rho}_{f_1}$ and $\bar{\rho}_{f_2}$ are irreducible $G_\Q$-modules and hence equivalent as $G_\Q$-representations. So we conclude that $A_{f_1}[2]\cong A_{f_2}[2]$ as $G_\Q$-modules. { Using SAGE, we compute that the field $K(\EC_1[2])$ is ramified at $2$}, therefore $\bar{\rho}_{\EC_1}$ is ramified at $2$, i.e. condition $\mathrm(ii)(c)$ of theorem \ref{mainthm} satisfies. Then $f_1$ and $f_2$ satisfy all the hypotheses of theorem \ref{mainthm} and applying the theorem, we deduce that  $f_1$ and $f_2$ are Greenberg $2$-Selmer companion modular forms over $K=\Q(\sqrt{-3})$ i.e. for every quadratic character $\chi$ of $K$, we have an isomorphism 
	$S_{Gr}(A_{f_1\chi}[2]/K)\cong S_{Gr}(A_{f_2\chi}[2]/K).$  
	
    {In general, we can show $f_1$ and $f_2$ are Greenberg $2$-Selmer companion  forms    over any imaginary quadratic field $K=\Q(\sqrt{-d})$, where $d$ is positive, square-free, $79\nmid d$  and $d\equiv 3\mod(4)$.}
	
	%%%%%%%%%%%%%%%%%%%%%%%%%%%%%%%%%%%%%%%%%%%%%%%%%%%%%%%%%%%%%%%%%%%%%%%%%%%%%%%%%%%%%%%%%%%%%%%%%%%%%%%%%%%%%%%%%%%%%%%%%%%%%%%%%%%%%%%%%%%%%%%%%%%%%%%%%%%%%%%%%%%%%
	\section{2-Selmer Near-Companion forms}\label{converse}
Fix $p=2$, an odd positive integer $N$, an even positive integer $k$ and  a number field $K$. 
	As before, $f_1, f_2\in S_k(\Gamma_0(Np^t))$ be two normalized cuspidal Hecke eigenforms which are newforms of level $Np^t$ for $t \geq 0$ and put $\cK:=\cK_{f_1, f_2}$. Let $\pi\mid p$ be a prime of $\cK$ induced by the embedding $i_p$. We fix an uniformizer $\pi$ of the ring of integers $\cO=\cO_{\cK_\pi}$ of $\cK_\pi$. 
	For $\dagger\in \{\text{BK, Gr}\}$, we now extend the notion of  near-companion elliptic curves \cite{mr} to $\dagger$ $\pi$-Selmer near-companion modular forms. 
  
	\begin{definition}\label{BK Sel nr com.}
	Let $f_1, f_2\in S_k(\Gamma_0(Np^t))$ be two newforms as above. Let $\dagger=BK$ or  $\dagger=Gr$ with $f_1$ and $f_2$ are $p$-ordinary. {We say $f_1$ and $f_2$ are $\dagger$ $\pi$-Selmer near-companion over $K$, if there exists a constant $C:=C(f_1, f_2, K)> 0$ such that for every quadratic character $\chi$ of $K$ and $0\le j\le k-2$,  \[ |\dim_{\F_2}S_\dagger(A_{f_1\chi(-j)}[\pi]/K)-\dim_{\F_2}S_\dagger(A_{f_2\chi(-j)}[\pi]/K)|<C.\]}
	\end{definition}
   We have the following variant of theorem \ref{mainthm} for Greenberg $\pi$-Selmer near-companion forms. This is a generalization of \cite[Proposition 7.13]{mr}. Notice that we only assume that $f_1$ and $f_2$ are congruent mod $\pi$ in proposition \ref{mainthm gr nr com.}. 
	\begin{prop}\label{mainthm gr nr com.}
		Let $f_1,\, f_2\in S_k(\Gamma_0(Np^t))$ be two $p$-ordinary  newforms. Assume that hypothesis $\mathrm{\ref{I}}$ holds. Suppose that there exist a $G_K$-isomorphism $A_{f_1}[\pi]\cong A_{f_2}[\pi]$, then $f_1$ and $f_2$ are Greenberg $\pi$-Selmer near-companion forms over $K$.   
	\end{prop}
    \begin{proof} The proof is similar to \cite[Proposition 7.13]{mr}, using lemma \ref{Lemma 2}$(i), (ii)$. 
	\end{proof}
	For the rest of \S \ref{converse}, we make the following hypothesis, which is a stronger version of hypothesis \ref{I}.
	\begin{customhypo}{III}\label{III}
		For the prime $\pi \mid 2$, $\cK_\pi=\Q_2$ holds.    
	\end{customhypo}
Now we fix  $f_1, f_2\in S_k(\Gamma_0(Np^t))$ as above.
We first establish an analogue of theorem \ref{Gr nr com. converse} for the Bloch-Kato $\pi$-Selmer group. This is an auxiliary result which is used in the proof of theorem \ref{Gr nr com. converse}.   
	\begin{prop}\label{BK-near-mainthm}
	Let $p=2$, $N \in \NN$ be odd, $k \in \NN$ be even and $K$ be a number field.	{Let $f_1, f_2\in S_k(\Gamma_0(Np^t))$ be two newforms} such that hypothesis \ref{III} holds. Assume that  either one of the  residual Galois representations  $\bar{\rho}_{f_1}$ or $\bar{\rho}_{f_2}$ is an irreducible $G_K$-module. Suppose that $f_1$ and $f_2$ are Bloch-Kato $\pi$-Selmer near-companion over $K$, then there is a $G_K$-isomorphism, $A_{f_1}[\pi]\cong A_{f_2}[\pi].$
	\end{prop} 
    The proof of proposition \ref{BK-near-mainthm} is prepared via lemmas \ref{set X}-\ref{modular, ecc sel com}. In fact, we begin our proof by making the following observation.
    
	 Notice that,  $A_{f_i}[\pi]\cong \frac{\ZZ}{2\ZZ}\oplus\frac{\ZZ}{2\ZZ}$ as a $\frac{\ZZ}{2\ZZ}$-vector space and let $K_i=K(A_{f_i}[\pi])$ denotes the fixed field of $\ker \bar{\rho}_{f_i}$, for $i=1, 2$. Since $\Gal(K_i/K)\hookrightarrow \Aut(A_{f_i}[\pi])\cong S_3$, it follows that $K_1=K_2$ if and only if $A_{f_1}[\pi]\cong A_{f_2}[\pi]$  as $G_K$-module (see \cite[Lemma 1.5]{Myu}).  
     Let $g_i(X)$ be the cubic irreducible polynomial over $K$ such that $K_i$ is the splitting field of $g_i(X)$, for $i=1,2 $. Consider the elliptic curve $\EC_i$ defined by  $\EC_i:Y^2=g_i(X)$. Then $K_i=K(\EC_i[2])$ and  $A_{f_i}[\pi]\cong \EC_i[2]$ as $G_K$-modules.  Let $S$ denote the finite set of primes of $K$ dividing $2N\Delta_1\Delta_2\infty$, where $\Delta_i$ is the discriminant of $\EC_i$ over $K$. Without any loss of generality, we assume that $\bar{\rho}_{f_2}$ is an irreducible $G_K$-module, which implies that $3\mid [K_2 : K]$. Let $F$ denote the compositum of the fields $K(A_{f_1}[\pi^2])$ and $K(\EC_1[4])$. Consider the restriction map $\Sel_2(\EC_1/K)\subset H^1(G_K, \EC_1[2])\rightarrow \Hom(G_{K_1}, \EC_1[2])$ by $s\mapsto \tilde{s}$. Let $L_1$ denotes the fixed field of $\underset{s\in \Sel_2(\EC_1/K)}{\cap}\ker \tilde{s}$ and $T_1$ denotes the Galois closure of $L_1$ over $K$. Set $T:=T_1F$. 
     
     The proof of the following lemma is similar to  \cite[Theorem 4.4]{Myu} and is omitted. 
	\begin{lemma}\label{set X}
		Suppose that $K_1\ne K_2$, then there exist infinitely many primes $v\notin S$ such that $\Frob_v|_T=1$  and $\Frob_v|_{K_2}$ has order $3$.    
	\end{lemma}
	
 We define the (infinite) set $X:=\{q \text{ prime of } K: \Frob_q|_T=1 \text{ and }\Frob_q|_{K_2} \text{ has order 3}\}$. Also define $P_0:=\{q\notin S:\EC_i(K_q)[2]=0 \tx{ for }i=1,2\}$. Now we have two possibilities  $3\mid [K_1: K]$ or $3\nmid[K_1 :K]$. Observe that,  if $3\nmid [K_1:K]$  then $P_0$ is the empty set.  In either case, given a prime $q\in X$, there is a quadratic character $\chi$ of $K$  with the following properties (see \cite[Proposition 4.3]{Myu}): \begin{enumerate}
		\item $\chi_v:=\chi_{|_{G_v}}=1_v$ for every $v\in S$.
        \item $\chi$ is ramified  at $q$.
		\item $\chi$ is unramified outside $\{q\}\cup P_0$.
	\end{enumerate}
{Let  $\mathfrak{X}$ be a  set  of  quadratic characters of $K$,  defined as follows:   \begin{equation}\label{def X}
	    \mathfrak{X}:=\{\chi\in \Hom(G_K, \{\pm 1\}):\chi \text{ is unramified outside } X\cup P_0 \text{ and is ramified at some subset of $X$}\}.
	\end{equation}}
	
    \noindent Now following the proof of \cite[Theorem 4.4]{Myu}, we obtain: 
	\begin{prop}\label{ecc sel com}
		Suppose that $K_1\ne K_2$, then for every positive integer $d$ there is a quadratic character  $\chi\in \mathfrak{X}$ such that, 
	$\dim_{\F_2}\Sel_2(\EC_1^\chi/K)-\dim_{\F_2}\Sel_2(\EC_2^\chi/K)>d.$\qed 
	\end{prop}
	
	Let $v\nmid 2N\infty$ be a prime of $K$  and $\chi$ be any quadratic character of $K$. { If $\chi$ is ramified, then $V_{f_{i}\chi}^{I_v}=0$ and hence $V_{f_i\chi}^{G_v}=\{0\}$, for $i=1,2$. On the other hand, if $\chi$ is unramified, since the complex absolute value of the eigenvalues of the action of $\Frob_v$ on $V_{f_i\chi}$ is not equal to $1$, we again have $V_{f_i\chi}^{G_v}=\{0\}$,} for $i=1,2$.  So we obtain $H^1(G_v/I_v, V_{f_i\chi}^{I_v})=0$ and this in turn implies that $H^1_{BK}(G_v, A_{f_i\chi}[\pi])=Im(\kappa_v)$, where $\kappa_v : A_{f_i\chi}^{G_v}/\pi\rightarrow H^1(G_v, A_{f_i\chi}[\pi])$ is the Kummer map. For a quadratic character $\chi$ of $K_v$, put $G_v^\chi:=\ker \chi$ and $G_{v, \chi}:=G_v/G_v^\chi$. In the following lemma, we determine $H^1_{BK}(G_v, A_{f_i\chi}[\pi])$ in terms of $A_{f_i\chi}[\pi]$.   
	\begin{lemma}\label{Lemma 7}
		Suppose that $K_1\ne K_2$ and $v$ be a prime of $K$ with $v\not\in S$. Then we have the following:
		\begin{enumerate}
				\item[(i)] If $\chi \in \mathfrak{X}$ is unramified at $v$, then for $i=1, 2$, $H^1_{BK}(G_v, A_{f_i\chi}[\pi])=Im(\kappa_v)=H^1_{ur}(G_v, A_{f_i\chi}[\pi])$. 
	\item[(ii)] If $\chi \in \mathfrak{X}$  is ramified at $v$, i.e. $v\in X\cup P_0$, then $H^1_{BK}(G_v, A_{f_2\chi}[\pi])=Im(\kappa_v)=0$
	\item[(iii)]  Let $\chi \in \mathfrak{X}$ be ramified at $v$, i.e. $v\in X\cup P_0. $ If $v\in X$, then $H^1_{BK}(G_v, A_{f_1\chi}[\pi])=Im(\kappa_v)=H^1(G_{v, \chi}, A_{f_1\chi_v}[\pi])$. However, if $v\in P_0$, then $H^1_{BK}(G_v, A_{f_1\chi}[\pi])=Im(\kappa_v)=0.$
		\end{enumerate}
	\end{lemma}
	\noindent$Proof:-$ $(i)$ Let $f\in \{f_1, f_2\}$. We know that for $v\notin S$, $V_{f\chi}^{G_v}=0$ and hence $H^0(G_v/I_v, A_{f\chi}^{I_v})$ is finite. Since $A_{f\chi}^{I_v}= A_{f}$ is divisible, it follows that $H^1(G_v/I_v, A_{f\chi}^{I_v})$, being both divisible and finite, vanishes. Consider the commutative diagram, 
 {\small   \begin{eqnarray}\label{eq17}
				\begin{tikzcd}
					0\arrow[r, ""] & A_{f\chi}^{G_v}/\pi\arrow[r, "\kappa_v"]\arrow[d, ""] &  H^1(G_v, A_{f\chi}[\pi])\arrow[r, ""]\arrow[d, ""] &H^1(G_v, A_{f\chi}) \arrow[d, ""] \\
					0\arrow[r, ""] & 0\arrow[r, ""] & H^1(I_v, A_{f\chi}[\pi])\arrow[r, ""] & H^1(I_v, A_{f\chi}). 
				\end{tikzcd}
    \end{eqnarray} }
    By {applying } {the} snake lemma in \eqref{eq17}, we deduce that $H^1_{BK}(G_v, A_{f\chi}[\pi])=Im(\kappa_v)=H^1_{ur}(G_v, A_{f\chi}[\pi])$.\\ 
	\noindent$(ii)$ Since $V_{f_2\chi}^{G_v}=0$, it follows that $A_{f_2\chi}^{G_v}$ is finite. This implies that $A_{f_2\chi}^{G_v}/\pi$ and $A_{f_2\chi}^{G_v}[\pi]$ have the same cardinality. Since $v\in X\cup P_0$ i.e. $\Frob_v|_{K_2}$ has order $3$, it follows  that $A_{f_2\chi}^{G_v}[\pi]=0$.  Therefore $A_{f_2\chi}^{G_v}/\pi=0$ and hence  $H^1_{BK}(G_v, A_{f_2\chi}[\pi])=Im(\kappa_v)=0$. \\
	\noindent$(iii)$ First we take $v\in X$. Applying {the} snake 
 lemma on the following commutative diagram 
 {\small   \begin{eqnarray}
        \begin{tikzcd}
					0 \arrow[r,] & A_{f_1\chi}^{G_v}/\pi \arrow[r, "\kappa_v"]\arrow[d, "\alpha_v"] & H^1(G_v, A_{f_1\chi}[\pi]) \arrow[r, ""]\arrow[d, ""] & H^1(G_v, A_{f_1\chi})\arrow[d, ""] & \\
					0\arrow[r, ]  & A_{f_1\chi}^{G_{v}^\chi}/\pi \arrow[r, "\kappa_v"] & H^1(G_v^\chi, A_{f_1\chi}[\pi])\arrow[r, ""] & H^1(G_v^\chi, A_{f_1\chi}), 
				\end{tikzcd}
    \end{eqnarray} }
    we  get the  exact sequence, 
\begin{eqnarray}\label{eq18}
    0\rightarrow\ker\alpha_v\rightarrow H^1(G_{v, \chi}, A_{f_1\chi}^{G_v^{\chi}}[\pi])\stackrel{i_v}{\rightarrow} H^1(G_{v, \chi}, A_{f_1\chi}^{G_v^{\chi}}).
\end{eqnarray}
	Let $\sigma$ be the {generator} of  $G_{v, \chi}$, $\xi\in  H^1(G_{v, \chi}, A_{f_1\chi}^{G_v^{\chi}}[\pi])$ be a cocycle and write $\xi(\sigma)=x \in A_{f_1\chi}^{G_v^{\chi}}[\pi]$. Since $\chi$ is trivial modulo $\pi$ and $\Frob_v$  acting trivially on $A_{f_1}[\pi]$, it follows that $H^1(G_{v, \chi}, A_{f_1\chi}^{G_v^{\chi}}[\pi])=\Hom(G_{v, \chi}, A_{f_1\chi}[\pi])$. Since $\Frob_v$ acts trivially on $A_{f_1}[\pi^2]$ as well, we get that $A_{f_1\chi}[\pi^2]=(A_{f_1\chi}[\pi^2])^{G_v^{\chi}}.$ Therefore there exists a $y\in (A_{f_1\chi}[\pi^2])^{G_v^{\chi}}\subset A_{f_1\chi}^{G_v^{\chi}}$ such that $x=\pi y$. Thus $\xi(\sigma)=x=2y=-2y=\chi(\sigma)y-y=\sigma y-y$. This shows that $\xi$ is a coboundary in $H^1(G_{v, \chi}, A_{f_1\chi}^{G_v^{\chi}})$ and hence $i_v$ in \eqref{eq18} is a zero map. Hence $\ker \alpha_v \cong H^1(G_{v, \chi}, A_{f_1\chi}[\pi])$ and  $\dim_{\F_2}\ker\alpha_v=\dim_{\F_2}H^1(G_{v, \chi}, A_{f_1\chi}[\pi])=2$. Since $\ker\alpha_v\hookrightarrow Im(\kappa_v)$ and $\dim_{\F_2}Im(\kappa_v)=2$, we conclude that $Im(\kappa_v)=H^1(G_{v, \chi}, A_{f_1\chi}[\pi]).$ 
    Next, assume that $v\in P_0$ i.e. $A_{f_1\chi}^{G_v}[\pi]=0$. Then proceeding in a similar way as in 
 $(ii)$, we get that $Im(\kappa_v)=H^1_{BK}(G_v, A_{f_1\chi}[\pi])=0.$ \qed 

 \medskip

 \noindent Next we compare the Bloch-Kato $\pi$-Selmer groups of $f_i\chi$ with $\Sel_2(\EC_i^\chi/K)$ for $i=1,2$.    
	\begin{prop}\label{modular, ecc sel com}
	Let $f_1, f_2$ and $\EC_1, \EC_2$ are as above. Then for every $\chi\in \mathfrak{X}$ and $i=1, 2$,  there exists constants $C_i:=C_i(f_i, K)>0$, independent of $\chi$ such that,
		\[|\dim_{\F_2}S_{BK}(A_{f_i\chi}[\pi]/K)-\dim_{\F_2}\Sel_2(\EC_i^\chi/K)|\le C_i.\]
	\end{prop}
	   
	\noindent$Proof:$ Put $\Omega:=\{v \tx{ prime in }K:v\mid \Delta_1\Delta_2N2\infty\}$ and take $\chi\in \mathfrak{X}$. Define
\begin{small}	\begin{eqnarray}\label{omega f}
		S^\Omega_{BK}(A_{f_1\chi}[\pi]/K):=\ker \Big( H^1(G_K, A_{f_1\chi}[\pi])\rightarrow\bigoplus_{v\notin\Omega}\frac{H^1(G_v, A_{f_1\chi}[\pi])}{H^1_{BK}(G_v, A_{f_1\chi}[\pi])}\Big).
	\end{eqnarray}
    \end{small}
	This gives the following exact sequence, 
\begin{small}	\begin{eqnarray}\label{exctseq1}
		0\rightarrow S_{BK}(A_{f_1\chi}[\pi]/K)\rightarrow S^\Omega_{BK}(A_{f_1\chi}[\pi]/K)\rightarrow\bigoplus_{v\in\Omega}\frac{H^1(G_v, A_{f_1\chi}[\pi])}{H^1_{BK}(G_v, A_{f_1\chi}[\pi])}.
	\end{eqnarray}
     \end{small}
Recall the classical $2$-Selmer group of  an elliptic curve $\EC/K$	is defined by $$\Sel_2(\EC/K):=\ker \Big( H^1(G_K, \EC[2])\rightarrow\bigoplus_{\text{all prime }v}\frac{H^1(G_v, \EC[2])}{\kappa_v\big(\frac{\EC(K_v)}{2\EC(K_v)}\big)}\Big), $$ where \begin{small} $\kappa_v:\frac{\EC(K_v)}{2\EC(K_v)}\rightarrow H^1(G_v, \EC[2])$  \end{small} is induced by the multiplication by $2$ map on $E$. Next, we define
 \begin{small}   \begin{eqnarray}\label{omega E}
		 \Sel_2^\Omega(\EC_1^\chi/K):=\ker \Big( H^1(G_K, \EC_1^\chi[2])\rightarrow\bigoplus_{v\notin\Omega}\frac{H^1(G_v, \EC_1^\chi[2])}{\kappa_v\big(\frac{\EC_1^\chi(K_v)}{2\EC_1^\chi(K_v)}\big)}\Big).
        \end{eqnarray}
         \end{small}
         Then we have the following exact sequence 
\begin{small}         \begin{eqnarray}
          0\rightarrow \Sel_2(\EC_1^\chi/K)\rightarrow \Sel_2^\Omega(\EC_1^\chi/K)\rightarrow\bigoplus_{v\in\Omega}\frac{H^1(G_v, \EC_1^\chi[2])}{\kappa_v\big(\frac{\EC_1^\chi(K_v)}{2\EC_1^\chi(K_v)}\big)}   
         \end{eqnarray}
          \end{small}
{Define $\frac{1}{2}C_1:=\underset{v\in\Omega}{\Sigma}\dim_{\F_2}H^1(G_v, A_{f_1}[\pi])$ and since $A_{f_1}[\pi]\cong \EC_1[2]$ as $G_K$-module, it follows that $\frac{1}{2}C_1=\underset{v\in\Omega}{\Sigma}\dim_{\F_2}H^1(G_v, \EC_1[2])$. Since $\chi\equiv 1\pmod{\pi}$, we also have $\frac{1}{2}C_1=\underset{v\in\Omega}{\Sigma}\dim_{\F_2}H^1(G_v, A_{f_1\chi}[\pi])$ for any quadratic character $\chi$ of $K$.} For any prime $v\notin \Omega$, $(V_p\EC_1^\chi)^ {G_v}=0$.  
{Following the proof of lemma \ref{Lemma 7} $(i),(iii)$, for $\chi\in\mathfrak{X}$ and $v\notin\Omega$, we obtain  $\kappa_v\big(\frac{\EC_1^\chi(K_v)}{p\EC_1^\chi(K_v)}\big)=H^1_{ur}(G_v, \EC_{1}^\chi[p])$, if $\chi$ is unramified at $v$ and $\kappa_v\big(\frac{\EC_1^\chi(K_v)}{p\EC_1^\chi(K_v)}\big)=H^1(G_{v, \chi}, \EC_{1}^\chi[p])$, if $\chi$ is ramified at $v$}. Following these observations together with the fact that $A_{f_1\chi}[\pi]\cong\EC_1^\chi[2]$ as $G_K$-module, we deduce from lemma \ref{Lemma 7} $(i)$, $(iii)$ that $S^\Omega_{BK}(A_{f_1\chi}[\pi]/K)\cong\Sel_2^\Omega(\EC_1^\chi/K)$. It is now plain from \eqref{omega f}, \eqref{omega E} and the definition of $C_1$ that  
	$|\dim_{\F_2}S_{BK}(A_{f_1\chi}[\pi]/K)-\dim_{\F_2}\Sel_2(\EC_1^\chi/K)|\le C_1. $
	
The corresponding result for $f_2$ and $\EC_2$ can be deduced by defining $\frac{1}{2}C_2:=\underset{v\in\Omega}{\Sigma}\dim_{\F_2}H^1(G_v, A_{f_2}[\pi])$ and following a similar  argument as above, using lemma \ref{Lemma 7} $(i)$, $(ii)$. \qed
	
Now we are ready to complete the proof of the proposition \ref{BK-near-mainthm}. 	
	\begin{proof}[Proof of proposition \ref{BK-near-mainthm}] Let $f_1, f_2\in S_k(\Gamma_0(Np^t))$ be two newforms which satisfy all the conditions stated in proposition \ref{BK-near-mainthm}.
		 By given hypothesis  $f_1$ and $f_2$ are Bloch-Kato $\pi$-Selmer near-companion over $K$, i.e. for every quadratic character $\chi$ of $K$ and for every $j$ with $0\le j\le k-2$, there exist a constant $C>0$ independent of $\chi$, such that 
		$|\dim_{\F_2}S_{BK}(A_{f_1\chi(-j)}[\pi]/K)-\dim_{\F_2}S_{BK}(A_{f_2\chi(-j)}[\pi]/K)|<C.$ 
        
		In particular for $j=0$, we have that
		\begin{align*}
			\abs{\dim_{\F_2}\Sel_2(\EC_1^\chi/K)-\dim_{\F_2}\Sel_2(\EC_2^\chi/K)}\le \abs{\dim_{\F_2}S_{BK}(A_{f_1\chi}[\pi]/K)-\dim_{\F_2}\Sel_2(\EC_1^\chi/K)}+\\\abs{\dim_{\F_2}S_{BK}(A_{f_1\chi}[\pi]/K)-\dim_{\F_2}S_{BK}(A_{f_2\chi}[\pi]/K)}+\abs{\dim_{\F_2}S_{BK}(A_{f_2\chi}[\pi]/K)-\dim_{\F_2}\Sel_2(\EC_2^\chi/K)}.
		\end{align*}
	Then  from lemma \ref{modular, ecc sel com}, we get 
        \begin{eqnarray}\label{eq21}
          \abs{\dim_{\F_2}\Sel_2(\EC_1^\chi/K)-\dim_{\F_2}\Sel_2(\EC_2^\chi/K)}\le C_1+C+C_2,\quad \text{ for every $\chi\in \mathfrak{X}$.}  
        \end{eqnarray}
	
	Now, if possible, assume that $A_{f_1}[\pi]\ncong A_{f_2}[\pi]$ as $G_K$-module. Then  $K_1 \neq K_2$  and the equation \eqref{eq21} 
contradicts lemma \ref{ecc sel com}. This completes the proof of proposition \ref{BK-near-mainthm}. 
	\end{proof}
	%%%%%%%%%%%%%%%%%%%%%%%%%%%%%%%%%%%%%%%%%%%%%%%%%%%%%%%%%%%%%%%%%%%%%%%%%%%%%%%%%%%%%%%%%%%%%%%%%%%%%%%%%%%%%%%%%%%%%%%%%%%%
	
	\begin{remark}\label{rmrk 2}
    
	\textnormal{	Note that for any $q\in X$, there is a quadratic character  $\chi\in \mathfrak X$ such that $\dim_{\F_2}\Sel_2(\EC_1^\chi/K) = \dim_{\F_2}\Sel_2(\EC_1/K) + 2$, and $\dim_{\F_2}\Sel_2(\EC_2^\chi/K)=\dim_{\F_2}\Sel_2(\EC_2/K)$ (see \cite[Theorems 4.4 \& 5.2]{Myu}). Inductively, twisting by product of such quadratic characters $\chi\in \mathfrak{X}$, we deduce that $\dim_{\F_2}\Sel_2(\EC_1^\chi/K)-\dim_{\F_2}\Sel_2(\EC_1/K)$ can be made arbitrarily large while $\dim_{\F_2}\Sel_2(\EC_2^\chi/K)=\dim_{\F_2}\Sel_2(\EC_2/K)$ remains unchanged. By lemma \ref{modular, ecc sel com}, we have that for $i=1, 2$ and for $\chi\in\mathfrak{X}$, the difference between $\dim_{\F_2}S_{BK}(A_{f_i\chi}[\pi]/K)$ and $\dim_{\F_2}\Sel_2(E_i^\chi/K)$ remains bounded, independent of $\chi$. Therefore  $\dim_{\F_2}S_{BK}(A_{f_1\chi}[\pi]/K)-\dim_{\F_2}S_{BK}(A_{f_2\chi}[\pi]/K)$ can be made arbitrarily large as $\chi$ varies over $\mathfrak{X}$. }   
	\end{remark}
Using (the proof of) proposition \ref{BK-near-mainthm}, we can now complete the proof of theorem \ref{Gr nr com. converse}.	
	\begin{proof}[Proof of Theorem \ref{Gr nr com. converse}]
     
		Without any loss of generality, assume that $\bar{\rho}_{f_2}$ is an irreducible $G_K$-module. We prove the contrapositive statement:
		 suppose that $A_{f_1}[\pi]\ncong A_{f_2}[\pi]$ as $G_K$-modules. Then $K_1\ne K_2$ and  using proposition \ref{ecc sel com}, proposition \ref{modular, ecc sel com} and remark \ref{rmrk 2}, we get that for every positive integer $d$,  there exist infinitely many $\chi\in \mathfrak{X}$ such that,  \begin{align}\label{eq6}
			\dim_{\F_2}S_{BK}(A_{f_1\chi}[\pi]/K)-\dim_{\F_2}S_{BK}(A_{f_2\chi}[\pi]/K)>d.
		\end{align}

        We claim that: for every $v\notin S$ and $\chi\in\mathfrak{X}$, $H^1_{Gr}(G_v, A_{f_2\chi}[\pi])=H^1_{BK}(G_v, A_{f_2\chi}[\pi]).$ Let us first establish this claim. {As $\chi\in\mathfrak{X}$, if $\chi$ is ramified at $v$ then $v\in X\cup P_0$ and $\Frob_v{_{|_{K_2}}}$ has order $3$}. This gives us $H^1(G_v/I_v, A_{f_2\chi}^{I_v})=0$ and hence $Im(\kappa_v)=H^1_{Gr}(G_v, A_{f_2\chi}[\pi])=H^1_{BK}(G_v, A_{f_2\chi}[\pi]).$ On the other hand, if $\chi$ is unramified at $v$, then it follows that $V_{f_2\chi}^{G_v}=0$ and consequently, $H^1(G_v/I_v, A_{f_2\chi})=0$. So again, we have $H^1_{Gr}(G_v, A_{f_2\chi}[\pi])=H^1_{BK}(G_v, A_{f_2\chi}[\pi]).$ This complete the proof of the claim. 
        
     Now,  it is known  that $S_{BK}(A_{f_i\chi}[\pi]/K)$ is a subgroup of  $S_{Gr}(A_{f_i\chi}[\pi]/K)$ (see \cite[Theorem 3]{flach90}). Putting these things  together with the fact that $H^1_{BK}(G_v, A_{f_2\chi}[\pi])\subset H^1_{Gr}(G_v, A_{f_2\chi}[\pi])$ for every prime $v$ of $K$,   we get the following exact sequence, 
		\begin{eqnarray}\label{eq26}
		    0\rightarrow S_{BK}(A_{f_2\chi}[\pi]/K)\rightarrow S_{Gr}(A_{f_2\chi}[\pi]/K)\rightarrow\underset{v\in S}{\bigoplus} \frac{H_{Gr}(G_v, A_{f_2\chi}[\pi])}{H_{BK}(G_v, A_{f_2\chi}[\pi])}.
		\end{eqnarray}
        From \eqref{eq26}, we obtain that $\dim_{\F_2}S_{Gr}(A_{f_2\chi}[\pi]/K)-\dim_{\F_2}S_{BK}(A_{f_2\chi}[\pi]/K)\le C^\prime$, where\\ $C^\prime:=\sum_{v\in S}\dim_{\F_2}H^1(G_v, A_{f_2\chi}[\pi])$  is independent of $\chi\in \mathfrak{X}$. Therefore for every $\chi\in \mathfrak{X}$, we have
		\begin{align*}
			\dim_{\F_2}S_{Gr}(A_{f_1\chi}[\pi]/K)-\dim_{\F_2}S_{Gr}(A_{f_2\chi}[\pi]/K)\ge \dim_{\F_2}S_{BK}(A_{f_1\chi}[\pi]/K) -\dim_{\F_2}S_{BK}(A_{f_2\chi}[\pi]/K)-C^\prime
		\end{align*}
		Now using \eqref{eq6}, we deduce that for every positive integer $d$, there are infinitely many $\chi\in\mathfrak{X}$, with  $$\dim_{\F_2}S_{Gr}(A_{f_1\chi}[\pi]/K)-\dim_{\F_2}S_{Gr}(A_{f_2\chi}[\pi]/K)\ge d-C'.$$ This shows that $f_1$ and $f_2$ are not Greenberg $\pi$-Selmer near-companion forms over $K$.  
	\end{proof}

	\bibliographystyle{unsrt}
	\bibliography{2_sel_com}

\end{document}